\documentclass{amsart}

\usepackage[all]{xy}
\usepackage{amsmath}
\usepackage{hyperref}
\hypersetup{
    colorlinks=true,    
    linkcolor=blue,          
    citecolor=blue,      
    filecolor=blue,      
    urlcolor=blue           
}
\usepackage{tikz}
\usetikzlibrary{graphs,positioning,arrows,shapes.misc,decorations.pathmorphing}

\tikzset{
    >=stealth,
    every picture/.style={thick},
    graphs/every graph/.style={empty nodes},
}

\makeatletter
\@namedef{subjclassname@2020}{\textup{2020} Mathematics Subject Classification}
\makeatother

\tikzstyle{vertex}=[
    draw,
    circle,
    fill=black,
    inner sep=1pt,
    minimum width=5pt,
]
\usepackage[position=top]{subfig}
\usepackage[alphabetic,backrefs]{amsrefs}
\usepackage{amssymb}

\usepackage{kantlipsum} 

\calclayout

\usetikzlibrary{decorations.pathmorphing}
\tikzstyle{printersafe}=[decoration={snake,amplitude=0pt}]

\newcommand{\codim}{\operatorname{codim}}

\newcommand{\supp}{\operatorname{Supp}}

\renewcommand{\qq}{\mathbb{Q}}
\newcommand{\zz}{\mathbb{Z}}

\newcommand{\rr}{\mathbb{R}}

\newtheorem{introthm}{Theorem}

\newtheorem{theorem}{Theorem}[section]
\newtheorem{lemma}[theorem]{Lemma}
\newtheorem{proposition}[theorem]{Proposition}
\newtheorem{corollary}[theorem]{Corollary}

\theoremstyle{definition}

\newtheorem{notation}[theorem]{Notation}
\newtheorem{definition}[theorem]{Definition}
\newtheorem{example}[theorem]{Example}

\newtheorem{remark}[theorem]{Remark}

\theoremstyle{remark}

\numberwithin{equation}{section}

\begin{document}

\title[Termination of pseudo-effective 4-fold flips]{Termination of pseudo-effective 4-fold flips}

\author[J.~Moraga]{Joaqu\'in Moraga}
\address{UCLA Mathematics Department, Box 951555, Los Angeles, CA 90095-1555, USA
}
\email{jmoraga@math.ucla.edu}


\thanks{The author was partially supported by NSF research grants no: DMS-1300750, DMS-1265285 
and by a grant from the Simons Foundation; Award Number: 256202}

\begin{abstract}
Let $(X,B)$ be a pseudo-effective log canonical $4$-fold over an algebraically closed field of characteristic zero. 
We prove that any sequence of $(K_X+B)$-flips terminates.
In particular, any minimal model program for $(X,B)$ terminates.
\end{abstract}

\maketitle

\setcounter{tocdepth}{1}
\tableofcontents

\section*{Introduction}

One of the aims of algebraic geometry 
is the classification of smooth projective varieties. 
The perspective of birational geometry 
is that this classification should be achieved by first 
performing some natural surgeries to our smooth projective variety, 
called birational modifications. 
A birational transformation is an isomorphism on an open set of our variety
and so it keeps many of its intrinsic characteristics. 
The minimal model program (MMP)  is a process that allows us to 
algorithmically perform such surgeries. 
This attempts to generalize the theory of minimal surfaces to higher dimensions. 
Starting with a smooth projective variety $X$, we perform 
a sequence of birational transformations
\[
X \dashrightarrow X_1 \dashrightarrow X_2 \dashrightarrow 
\dots \dashrightarrow X_n 
\]
such that curvature of $X_n$ is improved.
More precisely, the birational maps $\pi_i\colon X_i\dashrightarrow X_{i+1}$ are either flips or divisorial contractions.
The former is a small surgery, i.e., an isomorphism outside a closed subset of codimension at least two, while the latter contracts prime divisors and decreases the Picard rank. 
The aim is that $X_n$ must be either a {\em minimal model}, i.e., $K_{X_n}$ is nef or it admits a {\em Mori fiber space structure} $X_n\rightarrow W$, i.e., $-K_{X_n}$ is Fano over $W$ and $\dim W < \dim X_n$. 
The varieties appearing in this process tend to be singular, however, the singularities of the MMP tend to be mild. 
Two of the main goals of the MMP is to show that this algorithm exists 
and that it terminates. 
To do so, one needs to show that flips exist and they terminate, i.e., 
any sequence of flips is indeed finite. 
While we know how to control the sequence of divisorial contractions via the Picard rank, we do not know how to control the sequence of flips.
So far, the most successful approach to control sequence of flips has been the study of the singularities of the varieties $X_i$. This is the approach that we follow in this paper.
Below, we explain some of the advances in the existence and termination of flips and explain the main result of this article.

As mentioned above, two of the main goals of the minimal model program 
are to prove the existence of flips
and the termination of a sequence of such birational transformations.
The existence 
of flips for terminal $3$-folds was
achieved by Mori in ~\cite{Mor88}.
Koll\'ar and Mori classified terminal $3$-fold flips~\cite{KM92}.
The existence and termination of log canonical $3$-fold flips were achieved by Shokurov~\cites{Sho92,Sho96}.
Many of these results are based on the work of Kawamata on $3$-fold singularities~\cite{Kaw88}.
This result was generalized to the log canonical
case by Shokurov~\cite{Sho96}.
All these proofs rely on a careful analysis of the flipping contractions for $3$-folds.
In dimension $4$,
Kawamata settled the existence of smooth $4$-fold flips in ~\cite{Kaw89}.
In~\cite{KMM87}, the authors proved the termination of terminal flips for $4$-folds.
In~\cite{Sho04}, Shokurov proved the existence of flips for $4$-folds.
In~\cite{Sho85}, Shokurov introduced the {\em difficulty function} to study extremal ray contractions on algebraic varieties.
In~\cite{Fuj04}, Fujino used a variant of the difficulty function to prove the termination of canonical $4$-fold flips.
Later, in~\cite{AHK} Alexeev, Hacon, and Kawamata 
proved the termination of many klt $4$-fold flips using a similar invariant to the one in~\cite{Sho04}.
In ~\cite{Sho92}, Shokurov introduced a special sequence 
of flips for the minimal model program of a pair, 
called {\em ordered flips}, or {\em flips with scaling}, and proved that
any sequence of flips with scaling for $4$-folds terminates.
The existence of flips in arbitrary dimension for klt pairs was finally achieved by
Birkar, Cascini, Hacon, and M\textsuperscript{c}Kernan in ~\cite{BCHM10}.
The authors also proved the termination of flips with scaling
for klt pairs $(X,B)$ with $B$ a big $\rr$-divisor.
In~\cites{Bir12,HX13}, the existence of log canonical flips was proved.
In ~\cite{Bir07}, Birkar proved the termination of any sequence
of flips for klt pairs with $K_X+B \sim_\rr E\geq 0$
assuming termination of flips in dimension $\dim X -1$
and the ACC conjecture for log canonical thresholds.
The ACC conjecture for log canonical thresholds was proved by Hacon, M\textsuperscript{c}Kernan, and Xu in~\cite{HMX14}.

In this article, we prove the termination of flips for pseudo-effective log canonical $4$-folds.

\begin{introthm}\label{termination}
Let $(X,B)$ be a projective log canonical $4$-fold over an algebraically closed field of characteristic zero.
Assume that $K_X+B$ is pseudo-effective. 
Then, any sequence of $(K_X+B)$-flips terminates.
\end{introthm}

A straightforward consequence is the termination of any minimal model program for a pseudo-effective log canonical $4$-fold.

\begin{introthm}\label{thm:min-models}
Let $(X,B)$ be a projective log canonical $4$-fold over an algebraically closed field of characteristic zero.
Assume that $X$ is $\qq$-factorial
and $K_X+\Delta$ is pseudo-effective.
Then, any minimal model program for $(X,B)$ terminates with a minimal model.
\end{introthm}

In Theorem~\ref{termination-gen}, we will prove a version of Theorem~\ref{termination} and Theorem~\ref{thm:min-models}
for generalized pairs.

A minimal model is assumed to be
$\qq$-factorial. This explains the $\qq$-factoriality assumption of Theorem~\ref{thm:min-models}. 

The main idea of this article is to use the existence of minimal models for pseudo-effective $4$-folds~\cite{Sho09} to induce a generalized pair structure on each pair in the sequence of flips.
In the case of generalized pairs,
the existence of minimal models 
for pseudo-effective $4$-folds is proved in~\cite{LT19}.
Hence, we will heavily use the language of generalized pairs. We mimic the strategy of~\cite{Bir12} and use the ACC for generalized log canonical thresholds~\cite{BZ16}*{Theorem 1.5} to prove that flips terminate.
To achieve this, we need to prove the termination of generalized Kawamata log terminal $3$-fold flips. 
This $3$-dimensional part of the article follows closely the proof of termination of $3$-fold flips due to Shokurov~\cite{Sho96}.
To achieve this, we will use the 
existence of generalized dlt modifications (see, e.g.~\cite{FS20}*{Theorem 2.9}).
The idea of using minimal models to prove the termination of flips has been introduced by Birkar~\cites{Bir12b,BH14}.

For the sake of the exposition, we will present the results
starting from the lower dimensional cases.
In Section~\ref{S2}, we prove the termination of generalized klt $3$-fold flips.
In Section~\ref{S3}, we prove the special termination of flips for $\qq$-factorial generalized dlt $4$-folds, i.e., flips are eventually disjoint from the locus of log canonical centers.
In Section~\ref{S5}, we prove the termination of flips for $\qq$-factorial pseudo-effective dlt $4$-folds.
In this section, we also prove the main theorem of this article: the termination of pseudo-effective $4$-fold flips.\\

To culminate the introduction,
we give a brief sketch of the proof of Theorem~\ref{termination}. \\

\textit{Sketch of the proof:} 
The proof is by contradiction.
Assume that there exists an infinite sequence of flips for $(X,B)$ with $K_X+B$ pseudo-effective.
By~\cite{Sho09}, we know that $K_X+B$ is numerically equivalent to a sum 
$N+P$, where $N$ is an effective divisor
and $P$ is the push-forward of a nef divisor on a higher birational model of $X$.
In other words, $N$ is a b-nef divisor (see Definition~\ref{def:b-nef}). 
We let 
\[
X\dashrightarrow X_1 \dashrightarrow X_2 \dashrightarrow X_3 \dashrightarrow \cdots 
\dashrightarrow X_i \dashrightarrow \cdots 
\] 
be the models in the sequence of flips.
For each $i$, we let $N_i$ (resp. $P_i$)
be the push-forward of $N$ (resp. $P$) to the model $X_i$.
On each model $X_i$, we will consider the invariant 
\[
\lambda_i:= {\rm glct}((X_i,B_i);N_i+P_i) 
\in \rr_{\geq 0} \cup \{\infty\}.
\]
We define this invariant in Definition~\ref{def:glct}.
This number measures how much we can add $N_i+P_i$ to $B_i$
so that the singularities of $(X_i,B_i+\lambda_i N_i+\lambda_i P_i)$ are still singularities of the MMP (see subsection~\ref{ss:glct}).
Note that $P_i$ may not be effective, so we need to work with generalized singularities (see Definition~\ref{def:gen-sing}).
By Proposition~\ref{prop:finiteness}, the value $\lambda_i$ is finite unless $K_{X_i}+B_i$ is nef.
Furthermore, by Proposition~\ref{monotonicity} the sequence $(\lambda_i)_{i\geq 1}$, is non-decreasing.
By~\cite{BZ16}*{Theorem 1.5}, the sequence $(\lambda_i)_{i\geq 1}$ satisfies the ascending chain condition, so it eventually stabilizes to a number $\lambda_\infty$.
In Definition~\ref{def:gen-sing}, we define the generalized log canonical centers: the most singular locus of the generalized pair.
If the flipping loci in the sequence of flips are eventually disjoint from the generalized log canonical centers of $(X_i,B_i+\lambda_\infty N_i+\lambda_\infty P_i)$, then we can consider the sequence of flips in the complement of this locus.
By doing so, we may increase $\lambda_\infty$.
Otherwise, we have an infinite sequence of flips so that the flipping loci in the sequence intersect some generalized log canonical center $V$ infinitely many times
(see Definition~\ref{def:gen-sing}).
The previous argument is worked in Section~\ref{S5}.

Thus, we need to prove that the sequence of flips terminates around $V$, i.e., 
it is eventually disjoint from $V$.
This problem is known as special termination (see Section~\ref{S3}).
Using the existence of generalized dlt modifications (see, e.g.,~\cite{FS20}*{Theorem 2.29}), we may assume that $(X_i,B_i+\lambda_\infty N_i+\lambda_\infty P_i)$ are gdlt.
A generalized dlt pair is a glc pair where the generalized log canonical centers are snc at the generic points (see Definition~\ref{def:gdlt}).
Using Proposition~\ref{adj}, we can endow $V$ with the structure of a generalized pair $(V,B_V+M_V)$.
By Proposition~\ref{qflipslcp}, we obtain an infinite sequence of ample quasi-flips for $V$.
An ample quasi-flip is a sort of negative birational map that may contract and extract divisors (see Definition~\ref{qfgen}).
The concept of ample quasi-flips is crucial for our proof.
In Example~\ref{ex:ample-quasi-flip}, we show how flips for plt pairs naturally induce ample quasi-flips along boundary divisors. 
To complete the contradiction, 
it suffices to prove that this sequence
of ample quasi-flips in $V$ terminates.
In Proposition~\ref{tercod1}, we prove that this sequence of quasi-flips terminates in codimension one, so it is eventually a sequence of flips for a $3$-dimensional generalized klt pair.
To prove this termination, in Section~\ref{S2}, we generalize some techniques from~\cite{Sho96} 
to the setting of generalized pairs. 
In Proposition~\ref{ter3folds}, we will prove that this sequence of flips is eventually $B_V$-trivial and $M_V$-trivial.
This means that eventually, it is a sequence of flips for a klt $3$-fold and the termination is known in this case (see, e.g.,~\cite{KM98}*{Theorem 6.17}).
This finishes the sketch of the proof.\\

\textit{Postscript:} During the referee process of this paper Guodu Chen and Nikolaos Tsakanikas proved Theorem~\ref{termination} using similar methods (see~\cite{CT23}).

\subsection*{Acknowledgements}
The author would like to thank 
Caucher Birkar,
Christopher Hacon, 
Stefano Filipazzi, 
Tommaso de Fernex,
and Yoshinori Gongyo.
The author would like to express his gratitude to the anonymous referees for the careful report and all their suggestions.

\section{Preliminaries and notation}
\label{section:preliminaries}

In this section, we recall classic results and introduce some notation.
We will follow the notation of standard references in 
algebraic geometry~\cites{Laz04a}
and the minimal model program~\cites{KM98,HK10,Kol13}.
Throughout this paper, we will work over an algebraically closed field $\mathbb{K}$ of characteristic zero.
We work with $\rr$-divisors on algebraic varieties unless otherwise stated.

\begin{definition}
We work on normal varieties $X$ projective over $Z$.
Thus, we assume the existence of a projective morphism $X\rightarrow Z$.
We may write $X/Z$ to denote the relative setting.
A projective morphism $\psi\colon X\rightarrow Z$ is called a {\em contraction} if $\psi_*\mathcal{O}_X=\mathcal{O}_Z$.
In particular, if $X$ is normal, then $Z$ is normal.
Let $\pi\colon Y\dashrightarrow X$ be a birational map, its {\em exceptional locus}, denoted by ${\rm Ex}(\pi)$, is the union of the locus where $\pi$ is not defined and the locus where $\pi$ is defined and is not an isomorphism.
\end{definition}

\begin{definition}
Let $f\colon X\rightarrow Z$ be a projective morphism.
We denote by $N_1(X/Z)$ the $\rr$-vector space generated by irreducible curves
$C\subset X$ with $f(C)={\rm point}$, modulo numerical equivalence. 
The {\em relative Picard rank}, denoted by $\rho(X/Z)$, is the dimension of the $\rr$-vector space $N_1(X/Y)$. 
\end{definition}

\begin{definition}
Let $f\colon X\rightarrow Z$ be a projective morphism.
Let $C$ be a curve $C\subset X$ with $f(X)={\rm point}$.
Let $W_C$ be the union of all the effective curves on $X$,
that maps to points in $Z$, and are numerically equivalent to $C$ over $Z$.
We say that $C$ is {\em movable in codimension one over $Z$} 
if the closed subset $\overline{W_C}$ has codimension at most one. 
\end{definition}

\begin{definition}
Let $X/Z$ be a normal variety projective over $Z$.
Let $D$ be an $\rr$-Cartier $\rr$-divisor on $X$.
We say that the divisor $D$ is {\em pseudo-effective} over $Z$ if it lies in the closure
of the cone of big divisors of $X$ over $Z$.
We say that a  divisor $D$ is {\em nef} over $Z$ if $D\cdot C\geq 0$
for every curve $C\subset X$ that is contracted on $Z$.
A divisor $D$ on $X$ is said to be {\em anti-nef} over $Z$ if $-D$ is a nef divisor over $Z$.
\end{definition}

\begin{definition}
Let $X/Z$ be a normal variety projective over $Z$.
Let $\pi \colon Y \rightarrow X$ be a projective birational morphism over $Z$ from a normal variety $Y$
and let $E$ be a prime divisor on $Y$.
We say that $c_E(X):=\pi (E)$ is the {\em center} of $E$ on $X$.
We may identify the class of prime divisors over $X$ with the class of {\em divisorial valuations} of the function field $\mathbb{K}(X)$.
The center on $X$ of a divisorial valuation is the center
of the corresponding prime divisor.
\end{definition}

\subsection{Generalized pairs}
In this subsection, we recall some standard definitions around generalized pairs.
We will use the language of birational divisors on algebraic varieties.
We refer the reader to~\cite{Cor07} for the language of b-divisors.

\begin{definition}
Let $X/Z$ be a normal variety projective over $Z$.
The {\em birational category} $\mathcal{C}(X/Z)$ is the category whose elements are normal birational models of $X$ projective over $Z$ and its arrows are birational morphisms over $Z$.
The elements in $\mathcal{C}(X/Z)$ are said to be {\em birational models} of $X$ over $Z$.
\end{definition}

\begin{definition}
Let $X/Z$ be a normal variety projective over $Z$.
A {\em birational $\rr$-divisor} $\bf{M}$ on $\mathcal{C}(X/Z)$ consists of the following data:
\begin{enumerate}
    \item A divisor ${\bf M}_Y$ on $Y$ for every variety $Y \in \mathcal{C}(X/Z)$, and 
    \item an equality ${\bf M}_Y=\pi_* {\bf M}_{Y'}$ for every projective birational morphism $\pi \colon Y'\rightarrow Y$ in $\mathcal{C}(X/Z)$.
\end{enumerate}
Here, $\pi_*$ is the push-forward of $\mathbb{R}$-Weil divisors.
The divisor ${\bf M}_Y$ is called the {\em trace} of ${\bf M}$ on $Y$.
In the case that the birational class is clear from the context, we will say that ${\bf M}$ is a {\em birational divisor} on $X$.
We will write b-divisor instead of birational divisor to shorten the notation.
The trivial b-divisor is the b-divisor with trace $0$ on each model $Y\in \mathcal{C}(X/Z)$.
In what follows, instead of birational $\rr$-divisor ${\bf M}$, we may just write b-divisor ${\bf M}$, for simplicity.
\end{definition}

\begin{definition}
Let $X/Z$ be a normal variety projective over $Z$.
Let $D$ be an $\rr$-Cartier divisor on $X$.
It defines a b-divisor ${\bf M}$ as follows: 
For each projective birational map $Y\dashrightarrow X$ over $Z$, we resolve the map obtaining two projective birational morphisms $\pi' \colon Y'\rightarrow Y$ and
$\pi \colon Y'\rightarrow X$. Then, we can define 
${\bf M}_{Y}:={\pi'}_* \pi^*(D)$.
The above b-divisor is called the {\em $\rr$-Cartier b-divisor} induced by $D$.
It is usually denoted by $\overline{D}$.
\end{definition}

\begin{definition}\label{def:b-nef}
Two b-divisors on $X/Z$ are said to be equal if they have the same trace on every model.
Let ${\bf M}$ be a b-divisor on $X$.
The b-divisor ${\bf M}$ is said to be {\em $\rr$-Cartier} if it equals
$\overline{D}$ for some $\rr$-Cartier $\rr$-divisor $D$ on a birational model $Y$ of $X$ over $Z$.
The b-divisor ${\bf M}$ is said to be a {\em $b$-nef divisor over $Z$} if moreover $D$ is nef on $Y$ over $Z$.
We say that $p {\bf M}$ is {\em Cartier}, or equivalently, that ${\bf M}$ has {\em Cartier index $p$} where it descends if ${\bf M}=\overline{D}$ and $pD$ is Cartier.

Let $U\subset X$ be a non-empty open subset.
Let ${\bf M}$ be a b-divisor on $X$.
We say that ${\bf M}$ {\em descends} onto $U$ over $Z$ if
the restriction to $U$ of its trace on $X$ is $\rr$-Cartier
and for every projective birational morphism $\pi\colon Y\rightarrow X$, we have that
\[
{\bf M}_Y|_{U_Y}=\pi^*({\bf M}_X)|_{U_Y},
\]
where $U_Y:=\pi^{-1}(U)$.
If $U=X$, then we just say that ${\bf M}$ descends on $X$ over $Z$.
\end{definition}

The following definition is adopted from~\cite{HL18}*{Definition 2.13}.
Its importance relies on the fact that it allows compressing a technical assumption that appears often in the article.

\begin{definition}
Let ${\bf M}$ be a b-nef divisor on $X/Z$. 
Let $X'\rightarrow X$ be a model where ${\bf M}$ descends.
We denote by $M'$ the trace of ${\bf M}$ on $X'$, i.e., we have that
$M'={\bf M}_{X'}$.
We say that ${\bf M}$ is a {\em nef $\qq$-Cartier b-nef divisor} if we can write
\[
M' \equiv_Z \sum_j \mu_j M'_j
\]
where each $M'_j$ is a $\qq$-Cartier nef divisor over $Z$ and 
each $\mu_j$ is a nonnegative real number.
We may say that ${\bf M}$ is {\em NQC} for short.
\end{definition}

\begin{definition}\label{genpair}
A {\em generalized sub-pair} is a triple $(X/Z,B+M)$ where $X/Z$ is a normal algebraic variety projective over $Z$, 
$K_X+B+M$ is $\rr$-Cartier, $B$ is an $\rr$-divisor on $X$, and $M$ is the trace on $X$
of a b-nef divisor $\mathbf{M}$ over $Z$.
This means that $M=\mathbf{M}_X$.
A generalized sub-pair is called a {\em generalized pair} if $B$ is an effective divisor.
We say that $B$ is the {\em boundary part} and $M$ is the {\em nef part} of the generalized pair $(X/Z,B+M)$.
We call the sum $B+M$ a {\em generalized boundary}.
We say that the generalized pair
$(X/Z,B+M)$ is {\em NQC} if the b-nef divisor {\bf M} is NQC.
A generalized pair $(X/Z,B+M)$ is said to be {\em $\qq$-factorial} if its underlying algebraic variety $X$ is $\qq$-factorial.
\end{definition}

\begin{notation}
Let $(X/Z,B+M)$ be a generalized pair.
We denote by $X'$ a birational model of $X$ over $Z$ where ${\bf M}$ descends.
We have a projective birational morphism $X'\rightarrow X$. 
We denote by $M'$ the trace of ${\bf M}$ on $X'$, i.e., $M'={\bf M}_{X'}$. 
Whenever we write $X'$ for a generalized pair $(X/Z,B+M)$, we mean a model where ${\bf M}$ descends 
and admits a projective birational morphism to $X$. 
Analogously, whenever we write $M'$, we mean the trace of the b-nef divisor on the model $X'$.
Observe that $X'$ can always be replaced by higher models if needed.

If $Y$ is a birational model of $X$ over $Z$, then we set $M_Y:={\bf M}_Y$, to avoid the excessive use of ${\bf M}$.
If the birational model of $X$ is denoted by some subscript or superscript, for instance 
$X',X^+,$ or $X_i$, then we will set
$M':={\bf M}_{X'}, M^+:={\bf M}_{X^+}$, or $M_i:={\bf M}_{X_i}$, respectively. 
This is done to avoid the excessive use of double subscripts.
This notation is concordant with such of $M'$ in the previous paragraph.
To avoid confusion,
we will often remind the reader of this notation when it is used.
\end{notation}

\begin{definition}\label{logresolutiongeneralized}
Given a generalized pair $(X/Z,B+M)$ and a projective birational morphism $Y \rightarrow X$ over $Z$, 
we say that $f$ is a {\em log resolution} of the generalized pair if the following conditions are satisfied:
\begin{enumerate}
    \item $Y$ is a smooth variety,
    \item the exceptional locus of $Y\rightarrow X$ is purely divisorial,
    \item the strict transform of $B$ plus the reduced exceptional divisor has simple normal crossing, and 
    \item $Y\rightarrow X$ factors through a model where ${\bf M}$ descends.
\end{enumerate}
In particular, $M_{Y}:={\bf M}_Y$ is a nef divisor over $Z$
whenever $Y \rightarrow X$ is a log resolution of the generalized pair.
When we work with log resolutions of generalized pairs, we may recall that $M_Y$ is nef over $Z$ on such model, 
i.e., that ${\bf M}$ descends on $Y$.
\end{definition}

\begin{definition}\label{def:gen-sing}
Let $(X/Z,B+M)$ be a generalized pair, let $\pi \colon Y\rightarrow X$ be a projective birational morphism over $Z$, and let $E$ be a prime divisor on $Y$. 
We can write 
\[
K_Y+B_Y+M_Y=\pi^*(K_X+B+M).
\]
Note that $B_Y$ is uniquely defined by the above equality, since 
$M_Y$ only depends on the b-divisor ${\bf M}$.
The {\em generalized log discrepancy} of the generalized pair $(X/Z,B+M)$ at $E$ is 
\[
a_E(X/Z,B+M)=1-{\rm coeff}_E(B_Y).
\]
We say that $E$ has generalized log discrepancy $a_E(X/Z,B+M)$ with respect to $(X/Z,B+M)$.
We may write {\em gld} instead of generalized log discrepancy to shorten the notation.
Given $\epsilon \in (0,1)$, we say that the generalized pair $(X/Z,B+M)$ is {\em generalized $\epsilon$-Kawamata log terminal} (resp. {\em generalized log canonical}) 
if all its generalized log discrepancies are greater than $\epsilon$ (resp. nonnegative).
We say that the generalized pair is 
{\em generalized klt} if it is generalized $0$-Kawamata log terminal.
As usual, we may write {\rm gklt} (resp. {\rm glc}) to abbreviate generalized Kawamata log terminal (resp. generalized log canonical).
We say that $(X/Z,B+M)$ is {\em generalized canonical} if all its generalized log discrepancies are at least $1$.

Let $(X/Z,B+M)$ be a generalized pair.
If the generalized log discrepancy at $E$ is non-positive (resp. zero), we say that $E$ is a {\em generalized non-klt place} (resp. {\em generalized log canonical place}) of the generalized pair.
Furthermore, we say that $c_E(X)$ is a {\em generalized non-klt center} (resp. {\em generalized log canonical center}) of the generalized pair.
The union of all the generalized non-klt centers of a generalized pair $(X/Z,B+M)$ is the {\em generalized non-klt locus}.
The {\em generalized log canonical locus} of a generalized pair is the biggest open subset on which $(X/Z,B+M)$ has generalized log canonical singularities.
\end{definition}

\begin{definition}
A divisorial valuation $E$ over the generalized pair $(X/Z,B+M)$ is called {\em terminal} if $a_E(X/Z,B+M)$ is strictly greater than $1$.
Otherwise, we say that $E$ is {\em non-terminal}.
A generalized pair $(X/Z,B+M)$ is said to be {\em generalized terminal} if all its exceptional divisorial valuations are terminal with respect to $(X/Z,B+M)$.
Note that for $(X/Z,B+M)$ to be terminal, we only check exceptional valuations, because $X$ itself carries infinitely many valuations with log discrepancy equal to one.
\end{definition}

\begin{remark}
If ${\bf M}$ is the trivial b-divisor in the above definitions, then $(X/Z,B)$ is a {\em  pair} in the usual sense of~\cite{KM98}.
Conversely, every pair can be considered as a generalized pair with trivial nef part.
If we work with pairs, we will drop the word {\em generalized} from the above definitions.
\end{remark}

The following lemma is a straightforward consequence of the negativity lemma.

\begin{lemma}\label{lem:mono-ld}
Let $(X/Z,B+M)$ be a generalized pair.
Assume that $X$ is $\qq$-factorial.
Let $E$ be a divisorial valuation over $X$.
Then, we have that $a_E(X/Z,B+M)\leq a_E(X/Z,B) \leq a_E(X/Z)$.
\end{lemma}

\begin{definition}\label{def:FT}
A contraction $X\rightarrow Z$ is said to be 
of {\em Fano type} if there exists a boundary $B$ on $X$ 
such that the following conditions are satisfied:
\begin{enumerate}
    \item $(X,B)$ has klt singularities, and
    \item $-(K_X+B)$ is nef and big over $Z$.
\end{enumerate}
\end{definition}

We recall that, if $X\rightarrow Z$ is a Fano type morphism,
then any MMP of any divisor on $X$ relative to $Z$ terminates
(see, e.g.,~\cite{BCHM10}*{Corollary 1.3.1}).
The following proposition is well-known to the experts
(see, e.g.,~\cite{Bir19}*{2.13.(7)}).

\begin{proposition}\label{prop:FT}
Let $X\rightarrow Z$ be a Fano type morphism.
Let $(X,B+M)$ be a generalized log canonical pair
with $-(K_X+B+M)$ nef over $Z$.
Let $X'\rightarrow X$ be a projective birational morphism.
Assume that all the prime divisors on $X'$ exceptional over $X$
have generalized log discrepancy in the interval $[0,1)$ with respect to $(X,B+M)$.
Then, the morphism $X'\rightarrow Z$ is of Fano type.
\end{proposition}

\subsection{Quasi-flips for generalized pairs}
In this subsection, we recall the definition of quasi-flips for generalized pairs.

\begin{definition}\label{qfgen}
Let $(X/Z,B+M)$ be a generalized pair. A birational contraction $\phi \colon X \rightarrow W$ over $Z$ is said to be a {\em weak contraction} for
the generalized pair if $-(K_X+B+M)$ is nef over $W$.
If moreover $-(K_X+B+M)$ is ample over $W$,
then say that $\phi$ is an {\em ample weak contraction}.
A {\em quasi-flip} of $\phi$ is a birational map $\pi \colon X \dashrightarrow X^+$
with a birational contraction $\phi^+ \colon X^+ \rightarrow W$
such that the following conditions hold: 
\begin{enumerate}
\item The triple $(X^+/Z, B^{+}+ M^{+})$ is generalized log canonical, 
\item the $\rr$-Cartier $\rr$-divisor $K_{X^+}+B^{+}+M^{+}$ is nef over $W$, and
\item the inequality $\phi^{+}_* B^{+} \leq \phi_*B$ of effective $\rr$-divisors on $W$ holds.
\end{enumerate}
Here, $M^+$ is the trace of ${\bf M}$ on $X^+$.
As usual, the morphism $\phi$ is called the {\em flipping contraction}, and the morphism $\phi^+$ is called the {\em flipped contraction}.
The {\em flipping locus} is the exceptional locus of $\phi$
while the {\em flipped locus}
is the exceptional locus of $\phi^+$.
We say that the quasi-flip $\pi$ is {\em weak} if both $\phi$ and $\phi^+$ are small morphisms,
and that the quasi-flip $\pi$ is {\em ample} if both $-(K_X+B+M)$ and $(K_{X^{+}}+B^{+}+M^{+})$ are ample over $W$.
A {\em flip} for a generalized log canonical pair is a weak ample quasi-flip such that both $\phi$ and $\phi^+$ have relative Picard rank one.
Furthermore, in a flip, we assume that $\phi^+_* B^+ =\phi_* B$.
\end{definition}

\begin{example}\label{ex:ample-quasi-flip}
In this example, we show that flips for pairs $\pi \colon (X,V)\dashrightarrow (X^+,V^+)$ 
may induced ample quasi-flips $\pi_V \colon V \dashrightarrow V^+$ along the boundary divisor. This example enlightens the importance of studying ample quasi-flips.
We use the notation from toric geometry (see, e.g.,~\cite{CLS}). The example is based on the work on toric terminal $3$-fold flips due to Fujino, Sato, Takano, and Uehara (see~\cite{FSTU}).
Let $\sigma\subset \qq^3$ be the rational polyhedral cone spanned by
$\{e_1,e_2,e_3,v\}$,
where $v=(1,1,-r)$. Here $r\geq 1$. Then, the associated toric singularity $W=X(\sigma)$ 
is a non-$\qq$-factorial toric singularity. 
We consider two fans $\Sigma_1$ and $\Sigma_2$ such that 
$\Sigma_1(1)=\Sigma_2(1)=\sigma(1)$. $\Sigma_1$ is defined to be the fan whose cones are all the sub-cones of the $3$-dimensional cones $\langle e_1,e_2,v\rangle$ and $\langle e_1,e_2,e_3\rangle$.
$\Sigma_2$ is defined to be the fan whose cones are all the sub-cones of the $3$-diomensional cones $\langle e_1,e_3,v\rangle$ and $\langle e_2,e_3,v\rangle$.
The varieties $X:=X(\Sigma_1)$ and $X^+=X(\Sigma_2)$ are projective over $W$ and we have a small birational map
\begin{equation}\label{eq:flip}
\xymatrix{
X \ar@{-->}[rr]^-{\pi}\ar[rd]_-{\phi} &  & X^+ \ar[ld]^-{\phi^+} \\
 &  W &
}
\end{equation}
For $r=1$, the small birational map $\pi\colon X\dashrightarrow X^+$ in~\eqref{eq:flip} is the well-known Atiyah flop. 
For $r\geq 2$, the small birational map $\pi\colon X\dashrightarrow X^+$ in~\eqref{eq:flip} is a $K_X$-flip
and we have $K_X\cdot C=-1+\frac{1}{r}$ and $K_{X^+}\cdot C^+=r-1$ where $C$
and $C^+$ are the flipping and flipped curve, respectively.
Let $V$ be the $\qq$-Cartier divisor on $X$ corresponding to the ray $v$. 
Let $V^+$ be the strict transform of $V$ in $X^+$. 
Then, we have $(K_X+V)\cdot C = 1-2/r$ and $(K_{X^+}+V^+)\cdot C =2-r$. 
Thus, the birational map $\pi$ is a flip for the plt pair 
$(X,V)$ provided that $r\geq 3$. In this case, $V$ has a unique singular point along its intersection with $C$ which is the normal cone over a rational curve of degree $r$. 
The morphism $\phi$ maps $V$ isomorphic ally onto its image $\phi(V)$ in $W$. 
We have an induced ample quasi-flip 
\begin{equation}\label{eq:qflip}
\xymatrix{
V \ar@{-->}[rr]^-{\pi_V}\ar[rd]_-{\phi_V} &  & V^+ \ar[ld]^-{\phi_V^+} \\
 &  \phi(V) &
}
\end{equation}
where $\phi_V$ is the identity and $\phi_V^+$ is the minimal resolution.
Note that ${\phi_V^+}^*(K_{\phi(V)})=K_{V^+}+(1-2/r)C^+$.
Thus, the birational map $\pi_V$ can be regarded as taking the minimal resolution, 
pulling back the canonical, and then decreasing the coefficients along the exceptional curve from $1-2/r$ to $0$. 
If $r=2$, then $\pi_V$ in~\eqref{eq:qflip} is the inverse morphism of a crepant resolution, and $\pi_V$ is a quasi-flop.
\end{example}

\begin{definition}
A sequence of quasi-flips for a generalized log canonical pair is said to be {\em under the set $\Lambda$}
if the coefficients of the boundary parts belong to $\Lambda$.
We say that a sequence of quasi-flips for a generalized log canonical pair is {\em under a set satisfying the descending chain condition} (or {\em under a DCC set} for short) if $\Lambda$ satisfies the descending chain condition.
\end{definition}

\begin{definition}\label{qflopgen}
We say that a quasi-flip $\pi$ is a {\em quasi-flop} if both $\pi$ and $\pi^{-1}$ are quasi-flips.
In this case, the flipping contraction (resp. flipped contraction) is called the {\em flopping contraction} (resp. {\em flopped contraction})
and the flipping locus (resp. flipped locus) is called the {\em flopping locus} (resp. {\em flopped locus}).
\end{definition}

\begin{definition}\label{definitionstrict}
Let $\pi \colon X \dashrightarrow X^+$ be a quasi-flip for a generalized pair $(X/Z,B+M)$.
Let $\phi$ be the flipping contraction and let $\phi^+$ be the flipped contraction.
We define the {\em non-flopping locus} to be the smallest Zariski closed subset $N$ of $W$ so that
$\pi$ is a quasi-flop over $W \setminus N$.
The non-flopping locus is closed due to Proposition~\ref{monotonicity} below.
We say that a quasi-flip is {\em strict} if the subvariety $N$ is non-empty.
\end{definition}

\begin{definition}
Let $(X/Z,B+M)$ be a generalized pair.
An ample weak contraction $\phi \colon X\rightarrow W$ over $Z$, 
of relative Picard rank one, such that
$K_W+\phi_* B+ M_W$ is $\rr$-Cartier is called a 
{\em $(K_X+B+M)$-divisorial contraction} over $Z$.
Indeed, the $\rr$-Cartier condition of the divisor
$K_W+\phi_* B+ M_W$ and the relative ampleness of $K_X+B+M$ imply that the exceptional locus
of the morphism $\phi \colon X\rightarrow W$ is purely divisorial.

A {\em minimal model program} for $(X/Z,B+M)$ over $Z$, also written a $(K_X+B+M)$-MMP over $Z$, 
is a sequence of flips and divisorial contractions over $Z$.
\end{definition}

\begin{definition}\label{kltflip}
A quasi-flip for a generalized pair 
is said to be {\em glc} if its flipping locus (and therefore its flipped locus) does not intersect the generalized non-glc loci of the generalized pair.
A quasi-flip for a generalized log canonical pair is said to be {\em gklt} if its flipping locus (and therefore its flipped locus) does not intersect the generalized non-klt loci of the generalized pair.
A quasi-flip for a generalized log canonical pair is said to be {\em generalized terminal} if its flipping 
and flipped locus do not contain the center of a generalized non-terminal valuation.
\end{definition}

The following proposition is well-known for pairs, see for example ~\cite{Sho04}.
We give a proof in the case of generalized pairs for the sake of completeness.

\begin{proposition}\label{monotonicity}
Let $\pi \colon X \dashrightarrow X^+$ be a quasi-flip
for the generalized pairs $(X/Z,B+M)$ and $(X^+/Z,B^+ + M^+)$. Let $\phi \colon X\rightarrow W$ be the flipping contraction.
Let $E$ be a prime divisor over $X$.
Then, we have that
\[
a_E(X/Z,B+M) \leq a_E(X^+/Z,B^+ + M^+)
\]
and such inequality is strict if the center of $E$ on $W$ is contained in the non-flopping locus. The non-flopping locus is always closed. 
Furthermore, if $\phi_*B =\phi^+_*B^+$, then the non-flopping locus of an ample quasi-flip is the union of the images of the flipping and flipped locus on $W$.
\end{proposition}

\begin{proof}
We can resolve the birational map $\pi$ and obtain two projective birational morphisms 
$p\colon Y\rightarrow X$ and $q\colon Y\rightarrow X^+$.
We may assume that $p$ is a log resolution of the generalized pair $(X/Z,B+M)$.
In particular, ${\bf M}$ descends onto $Y$.
Note that 
\[
a_E(X^+/Z,B^++M^+)-a_E(X/Z,B+M)=
{\rm coeff}_E(p^*(K_X+B+M)-q^*(K_{X^+}+B^++M^+)).
\]
Furthermore, the divisor 
\[
F:=p^*(K_X+B+M)-q^*(K_{X^+}+B^++M^+)
\]
is anti-nef over $W$.
Moreover, its push-forward to $W$ equals 
$\phi_*B-\phi^+_* B^+$, which is effective by the definition of quasi-flip.
By the negativity lemma~\cite{KM98}*{Proposition 3.39}, we conclude that $F\geq 0$.
This proves the first statement.

Let $N:=\phi(p(F))$.
By the negativity lemma~\cite{KM98}*{Proposition 3.39}, we know that 
the support of $F$ contains every fiber of 
$\phi \circ p$ over any point contained in $N$.
In particular, we deduce that
\[
a_E(X/Z,B+M)<a_E(X^+/Z,B^++M^+)
\] 
for every divisor $E$ whose center in $W$ is contained in $N$.
Indeed, in this case, $E$ would be contained in the support of $F$.

We claim that $N$ contains the non-flopping locus.
It suffices to prove that the complement of $N$ is contained in the flopping locus.
Let $w\in W$ be a point in the complement of $N$.
Then, for every curve $C\subseteq \phi^{-1}(w)$, by the projection formula, we have that 
\[
(K_X+B+M)\cdot C = 0.
\]
Thus, $\pi$ is a quasi-flop over the complement of $N$. Conversely, if $p\in N $, then there is a curve $C\subset p^{-1}(\phi^{-1}(p))$
such that $F\cdot C<0$ (see, e.g.,~\cite[Lemma 2.10]{Lai11}). Therefore, we conclude that
$p^*(K_X+B+M)\cdot C < q^*(K_{X^+}+B^+ + M^+)\cdot C$.
This implies that $p$ is contained in the non-flopping locus and so the non-flopping locus equals $N$ which is a closed subset.  

Finally, assume that $\pi$ is an ample quasi-flip.
We argue that $\phi({\rm Ex}(\phi))\cup \phi^+({\rm Ex}(\phi^+)) = N$.
If $p\in \phi({\rm Ex}(\phi))$, then we can find a curve 
$C\in \phi^{-1}(p)$ and $(K_X+B+M)\cdot C <0$ so $p$ is not in the flopping locus. Proceeding analogously for $\phi^+({\rm Ex}(\phi^+))$, we conclude that 
$\phi({\rm Ex}(\phi))\cup \phi^+({\rm Ex}(\phi^+)) \subset N$.
If $p$ is not in $\phi({\rm Ex}(\phi))\cup \phi^+({\rm Ex}(\phi^+))$, then both $\phi$ and $\phi^+$ are isomorphisms in a neighborhood of $p$, and so
$\pi$ is an isomorphism in a neighborhood of $p$. 
As we are assuming that $\phi_*B=\phi^+_*B^+$, 
we conclude that $p\in W$ is in the flopping locus. 
\end{proof}

\begin{remark}\label{rem} 
The previous proposition implies that given an ample quasi-flip
$\pi\colon (X/Z,B+M)\dashrightarrow (X^+/Z,B^+M^+)$,
if the log discrepancies of $(X/Z,B+M)$ equal such of $(X^+/Z,B^+ +M^+)$,
then $\pi$ is an isomorphism. 
In particular, a strict ample quasi-flip must change at least one 
log discrepancy.  
\end{remark}

\subsection{Minimal models}
In this subsection, we recall the standard definition of minimal models.
First, we recall the definition of generalized divisorially log terminal pairs.

\begin{definition}
Let $(X/Z,B+M)$ be a generalized pair.
We say that $(X/Z,B+M)$ is generalized divisorially log terminal if there exists an open set $U\subset X$ satisfying the following conditions:
\begin{enumerate}
    \item the coefficients of $B$ are $\leq 1$, 
    \item $U$ is smooth and $B|_U$ has simple normal crossing support, 
    \item every generalized non-klt center of $(X/Z,B+M)$ intersects $U$ non-trivially and is given by strata of $\lfloor B \rfloor$ i.e., it is given by intersection of prime components of $B$ with coefficient one, and 
    \item the b-divisor ${\bf M}$ descends over the generic point of each generalized non-klt center of $(X/Z,B+M)$.
\end{enumerate}
To shorten notation, we say that $(X/Z,B+M)$ is {\em gdlt}.
\end{definition}

\begin{definition}\label{definitionminimalmodel}
Let $(X/Z,B+M)$ be a $\qq$-factorial gdlt pair.
Assume that $X$ is projective.
Let $\pi \colon X \dashrightarrow Y$ be a birational map over $Z$, to a normal quasi-projective variety $Y$ which is projective over $Z$.
We say that $\pi$ is a {\em minimal model} of $(X/Z,B+M)$ if the following conditions hold:
\begin{enumerate}
    \item $\pi^{-1}$ contracts no divisors,
    \item $Y$ is $\qq$-factorial,
    \item $K_Y+B_Y+M_Y$ is nef over $Z$ where $B_Y:=\pi_*B$, and 
    \item $a_E(X/Z,B+M)<a_E(Y/Z,B_Y+M_Y)$ for any $\pi$-exceptional prime divisor $E\subset X$.
\end{enumerate}
When $\pi$ is clear from the context, we may say that  
$(Y/Z,B_Y+M_Y)$ is a {\em minimal model} of $(X/Z,B+M)$.
Here, as usual $M_Y$ is the trace of ${\bf M}$ on $Y$.
We say that $(Y/Z,B_Y+M_Y)$ is a good minimal model if $K_Y+B_Y+M_Y$ is semi-ample over $Z$.
\end{definition}

\subsection{Generalized log canonical thresholds}\label{ss:glct}
In this subsection, we recall the definition of the generalized log canonical threshold.

\begin{definition}\label{def:glct}
Let $(X/Z,B+M)$ be a generalized log canonical pair.
Let $N$ be an effective divisor on $X$ and ${\bf P}$ a $b$-nef divisor over $Z$.
By abuse of notation, we will denote by $P$ the trace of ${\bf P}$ on $X$.
Assume that $N+P$ is an $\rr$-Cartier divisor on $X$.
Let $U\subset X$ be a non-empty open subset of $X$.
We define the {\em generalized log canonical threshold} of $(X,B+M)$ with respect to $N+P$ over $U$ to be
\[
{\rm glct}((X/Z,B+M),N+P;U):=
\sup\{ \lambda \mid (X/Z,B+\lambda N+(M+\lambda P)) \text{ is generalized log canonical over $U$}\}.
\]
Here, $(X,B+\lambda N+(M+\lambda P))$ is considered to be a generalized pair with boundary part $B+\lambda N$ and nef part $M+\lambda P$.
If $U=X$, then we just call it the generalized log canonical threshold of
$(X,B+M)$ with respect to $N+P$
and denote it by ${\rm glct}((X/Z,B+M);N+P)$.
\end{definition}

\begin{proposition}\label{prop:finiteness}
Let $(X/Z,B+M)$ be a $\qq$-factorial generalized log canonical pair.
Assume that we can write 
$K_X+B+M\equiv_{Z} N+P$, where $N$ is effective and ${\bf P}$ is b-nef over $Z$.
Let $U\subset X$ be a non-empty open subset.
If ${\rm glct}((X/Z,B+M),N+P;U)$ is infinite, then
$N|_U=0$ and ${\bf P}$ descends over $U$.
In particular, $K_X+B+M$ intersects nonnegatively every projective curve
contained in $U$ which contracts on $Z$.
Furthermore, if $U=X$ and ${\rm glct}((X/Z,B+M),N+P)$ is infinite, then
$K_X+B+M$ is nef over $Z$.
\end{proposition}

\begin{proof}
If $N|_U$ is non-zero, then we can find a prime divisor $F_0$ on $X$ which intersects $U$
so that ${\rm coeff}_{F_0}(N)=\lambda>0$.
Then, we have that ${\rm glct}((X/Z,B+M),N+P;U)\leq \lambda^{-1}$,
leading to a contradiction.
Hence, we may assume that $N=0$, 
so we have that $P$ is $\rr$-Cartier.
Assume that $P$ does not descend over $U$.
Let $\pi\colon Y\rightarrow X$ be a projective birational morphism from a model $Y$ on which $P$ descends.
We may assume that $Y$ is $\qq$-factorial.
Denote by $P_Y$ the trace of ${\bf P}$ on $Y$.
Then, we can write $\pi^*P=P_Y+F$, where $F$ is an effective divisor exceptional over $X$.
Since ${\bf P}$ does not descend over $U$, then there is a prime divisor $F_1$ on $Y$ 
so that ${\rm coeff}_{F_1}(F)=\lambda>0$.
A simple computation yields 
\[
{\rm glct}((X/Z,B+M),N+P;U)\leq a_E(X,B+M)\lambda^{-1}.
\]
Thus, we conclude that if the generalized log canonical threshold is infinite, then
$N|_U=0$ and $P$ descends over $U$.

Now, we proceed to prove the second statement.
We prove that if the generalized log canonical threshold is infinite, then
$K_X+B+M$ intersects nonnegatively every projective curve contained in $U$ that is being contracted on $Z$.
Assume that the generalized log canonical threshold is infinite.
By the above paragraph, we have that $N|_U=0$ and ${\bf P}$ descends over $U$.
So, we have that 
\[
(K_X+B+M)|_U \equiv_{Z} P|_U.
\] 
Let $Y$ be a model where ${\bf P}$ descends.
Let $U_Y$ be the preimage of $U$ on $Y$.
Let $C$ be any projective curve contained in $U$ which is contracted on $Z$. 
Let $C_Y$ be a curve on $Y$ that maps surjectively to $C$.
Then, by the projection formula, we have that
\[
(K_X+B+M)\cdot C=
\pi^*(K_X+B+M)|_{U_Y} \cdot C_Y = P_Y|_{U_Y} \cdot C_Y \geq 0.
\]
The last inequality follows from the fact that $P_Y$ is nef over $U$
and $C_Y$ is being contracted on $Z$.

The last statement follows trivially from the above statement
by setting $U=X$.
\end{proof}

\subsection{Diminished base locus}
In this subsection, we recall and prove some statements about the diminished base locus.

\begin{definition} 
Let $X$ be a normal variety, 
$X\rightarrow Z$ a projective contraction, 
and $D$ a divisor on $X$.
We define the {\em relative diminished base locus} of $D$ over $Z$, to be 
\[
{\rm Bs}_{-}(D/Z):=\bigcup_{\lambda >0}{\rm Bs}(D+\lambda A/Z),
\]
where $A$ is an ample divisor over $Z$ and ${\rm Bs}(-/Z)$ is the base locus of the divisor relative over $Z$.
Observe that the relative diminished base locus is a countable union of subvarieties of $X$.
In the case that $Z$ is clear from the context, we may just call it the diminished base locus.
Note that the diminished base locus of $D$ is a proper subset of $X$
if and only if $D$ is pseudo-effective over $Z$.
Furthermore, the diminished base locus of $D$ is empty if
$D$ is nef over $Z$.
\end{definition}

\begin{lemma}\label{lem:dim-conv}
Let $X$ be a normal variety
and $X\rightarrow Z$ be a projective contraction.
Let $D$ and $F$ be pseudo-effective divisors on $X$.
Let $( \epsilon_i)_{i\geq 1}$ be a sequence of real numbers that converges to zero.
Then, we have that ${\rm Bs}_{-}(D/Z) \subset \bigcup_{i=1}^\infty {\rm Bs}(D+\epsilon_i F/Z)$.
\end{lemma}
\begin{proof} 
It suffices to prove that
for each $\lambda>0$, we can find $i$ large enough such that
\[
{\rm Bs}(D+\lambda A/Z)\subset {\rm Bs}(D+\epsilon_i F/Z).
\]
Since $A$ is ample,
we can choose $i$ large enough so that 
${\rm Bs}(\lambda A - \epsilon_i F/Z)=\emptyset$.
Hence, we have that
\[
{\rm Bs}(D+\lambda A/Z) 
= {\rm Bs}((\lambda A -\epsilon_i F)+(D+\epsilon_i F)/Z)
\subset {\rm Bs}(\lambda A-\epsilon_i F/Z)\cup {\rm Bs}(D+\epsilon_i F/Z)
= {\rm Bs}(D+\epsilon_i F/Z).
\]
This concludes the proof.
\end{proof}

\begin{lemma}\label{lem:dim-pull-back}
Let $X$ and $Y$ be normal quasi-projective varieties, both of them projective over $Z$.
Let $\pi\colon Y\rightarrow X$ be a projective birational morphism over $Z$.
Let $M$ be an $\rr$-Cartier divisor on $X$ such that its diminished base locus over $Z$ is disjoint from the proper closed subset $L\subset X$.
Then, the diminished base locus of $\pi^*M$ over $Z$ is disjoint from the proper closed subset $\pi^{-1}(L) \subset Y$.
\end{lemma}

\begin{proof}
Let $y\in Y$ be a point that is not contained in $\pi^{-1}(L)$. 
Then, its image $x=\pi(y)$ on $X$ is not contained in $L$.
We can find a Cartier divisor $A$ on $X$ ample over $Z$. We can find a sequence of effective divisors
$0\leq E_n \sim_{\rr,Z} M+\frac{1}{n} A$ such that 
$x$ is not contained in the support of $E_n$.
Pulling-back to $Y$, we get that
$0\leq \pi^*E_n \sim_{\rr,Z} \pi^*M +
\frac{1}{n}\pi^*A$. Furthermore, $\pi^*E_n$ does not contain $y$ in its support.
Thus, the union
$\cup_{n=1}^{\infty} {\rm Bs}(\pi^*M + \frac{1}{n}\pi^*A)$ is disjoint from $\pi^{-1}(L)$.
By Lemma~\ref{lem:dim-conv}, we conclude that ${\rm Bs}_{-}(\pi^*M)$ is disjoint from $\pi^{-1}(L)$. 
\end{proof}

\begin{lemma}\label{lem:dim-MMP}
Let $(X/Z,B+M)$ be a generalized pair.
Let $X\dashrightarrow X^+$ be a step of a $(K_X+B+M)$-MMP over $Z$.
Let $V\subset {\rm Bs}_{-}(K_X+B+M/Z)$ be an irreducible component 
that is not contained in the exceptional locus of $X\dashrightarrow X^+$.
Then, the strict transform of $V$ on $X^+$ is contained in ${\rm Bs}_{-}(K_{X^+}+B^{+}+M^{+}/Z)$.
Here, $M^+$ is the trace of the b-nef divisor ${\bf M}$ on $X^+$.
\end{lemma}

\begin{proof}
Let $V^+$ be the strict transform of $V$ on $X^+$.
Assume $V^+$ is not contained
in the diminished base locus of $K_{X^+}+B^{+}+M^{+}$ relative to $Z$.
Let $A^+$ be an effective Cartier divisor on $X^+$ ample over $Z$.
For every $n\geq 1$, we can find an effective divisor
$0\leq E^+_n \sim_{\rr,Z} K_{X^+}+B^{+}+M^{+}+\frac{1}{n}A^{+}$ such that 
$V^+$ is not contained in the support of $E^{+}_n$.
Let $E^0_n$ be the rational pull-back of $E^{+}_n$ to $X$.
We denote by $A$ the rational pull-back of $A^{+}$ to $X$.
Note that for $n$ large enough, the birational map $X\dashrightarrow X^+$ is also a step of a 
$(K_X+B+M+\frac{A}{n})$-MMP.
By Proposition~\ref{monotonicity}, we can write 
\[
\phi^*\left(K_{X^{+}}+B^{+}+M^{+}+\frac{A^{+}}{n} \right) = 
K_X+B+M+\frac{A}{n} - F_n,
\] 
where $\phi^*$ denotes the rational pull-back
and $F_n$ is an effective divisor supported on the exceptional locus of $X\dashrightarrow X^+$.
In particular, $F_n$ does not contain $V$ in its support.
On the other hand, we have that 
\[
K_X+B+M+\frac{A}{n} \sim_{\rr,Z} E_n:= E^0_n+F_n \geq 0.
\]
Thus, $V$ is not conatined in $\cup_{i=1}^n {\rm Bs}(K_{X}+B+M+\frac{1}{n}A/Z)$ for $n$ large enough.
By Lemma~\ref{lem:dim-conv}, we conclude that $V$ is not contained in the diminished base locus of $K_{X}+B+M$ relative to $Z$.
This gives a contradiction.
We conclude that $V^+$ is contained in the diminished base locus of $K_{X^{+}}+B^{+}+M^{+}$ relative to $Z$.
\end{proof}

\begin{lemma}\label{lem:dim-small}
Let $X'\rightarrow X$ be a projective birational morphism over $Z$.
Let $M'$ be an $\rr$-Cartier divisor on $X'$
nef over $Z$.
Let $M$ be the push-forward of $M'$ to $X$.
Assume that $X$ is $\qq$-factorial.
Then, the diminished base locus of $M$ over $Z$
has codimension at least two.
\end{lemma}

\begin{proof}
By the negativity lemma~\cite{KM98}*{Proposition 3.39}, we can write $\pi^*M=M'+E'$ where $E'$ is an effective divisor exceptional over $X$. We claim that the diminished base locus of $M$ is contained in 
$\pi({\rm Ex}(\pi))$.
Let $x\in X$ be a point that is contained in the isomorphism locus of $X'\rightarrow X$.
Let $x'\in X'$ be its pre-image on $X'$.
Let $A'$ be a Cartier divisor on $X'$ ample over $Z$.
Then, for every $n\in \zz_{\geq 1}$, we have that 
\[
{\rm Bs}\left(M'+E'+ \frac{1}{n} A'/Z \right) \subset 
{\rm Bs}\left(M'+\frac{1}{2n} A'/Z \right) \cup {\rm Bs}\left(E'+\frac{1}{2n}A'/Z \right)
\subset {\rm Bs}\left(E'+\frac{1}{2n}A'/Z\right) \subset \supp(E').
\]
The second containment holds as
$M'+\frac{1}{2n}A'$ is ample over $Z$,
so its diminished base locus relative to $Z$ is empty.
Thus, for each $n\in \zz_{\geq 1}$, we can find an effective divisor $E'_n \sim_{\rr,Z} M'+E'+\frac{1}{n} A'$ whose support does not contain $x'$.
Let $E_n$ be the push-forward of $E'_n$ to $X$.
Then, we have that the support of $E_n$ does not contain $x$.
Furthermore, we have that $E_n \sim_{\rr,Z} M + \frac{1}{n}A$,
where $A$ is the push-forward of $A'$ on $X$.
Thus, we conclude that $x$ is not contained in $\cup_{n=1}^\infty {\rm Bs}(M+\frac{1}{n} A/Z)$.
Since $x$ is an arbitrary point of $X$ that is not contained in $\pi({\rm Ex}(\pi))$, 
we conclude that $\cup_{n=1}^\infty {\rm Bs}(M+\frac{1}{n}A/Z)\subset \supp(\pi({\rm Ex}(\pi)))$.
By Lemma~\ref{lem:dim-conv}, we conclude that
\[
{\rm Bs}_{-}(M) \subset \cup_{n=1}^\infty {\rm Bs}\left(M+\frac{1}{n}A/Z \right)\subset \supp(\pi({\rm Ex}(\pi)))
\]
Note that $\supp(\pi({\rm Ex}(\pi)))$ has codimension at least two. This proves the lemma.
\end{proof}

\begin{lemma}\label{dim:movable-cod-1}
Let $(X/Z,B+M)$ be a generalized pair.
Assume $X$ is $\qq$-factorial.
Let $C$ be a curve on $X$ mapping to a closed point on $Z$.
Assume $C$ that is movable in codimension one over $Z$.
Then, we have that $M\cdot C\geq 0$.
\end{lemma}

\begin{proof}
Let $\pi\colon X'\rightarrow X$ be a log resolution of the generalized pair $(X/Z,B+M)$ where $M$ descends.
By the negativity lemma~\cite{KM98}*{Proposition 3.39}, we have that $\pi^*M=M'+E$ where $E$ is an effective divisor exceptional over $X$.
Since $C$ is movable in codimension one, then it is numerically equivalent to a curve that is not contained in the image of ${\rm Ex}(\pi)$ on $X$.
We may assume $C$ itself has this property.
Let $C'$ be the strict transform of $C$ on $X'$.
Then, we can compute 
$M\cdot C = \pi^*M \cdot C' = (M'+E)\cdot C' \geq 0$.
The last inequality holds 
since $M'$ is nef over $Z$ and $E$ is effective and does not contain $C'$
in its support.
\end{proof}

\begin{lemma}\label{lem:blow-up-M}
Let $(X/Z,B+M)$ be a generalized pair.
Assume that $X$ is $\qq$-factorial.
Let $C\subset X$ be a curve such that $M\cdot C<0$
and $C$ maps to a point on $Z$.
Let $\pi\colon Y\rightarrow X$ be a projective birational morphism that extracts a unique
divisor $E$ whose center on $X$ is $C$.
Then, we have that $a_E(X/Z,B+M)<a_E(X/Z,B)$.
\end{lemma}

\begin{proof}
By the negativity lemma~\cite{KM98}*{Proposition 3.39}, we can write $\pi^*M=M_Y+aE$, where $a\geq 0$.
It suffices to prove that $a>0$.
Assume it is not. Then, we have that $\pi^*M=M_Y$.
For every point $e\in E$, we can find a curve $C_Y$ passing through $e$ whose image on $X$ is $C$.
Note that the image of $C_Y$ on $Z$ is a point.
By the projection formula, we have that $\pi^*M\cdot C_Y=M_Y\cdot C_Y <0$.
Thus, we conclude that every point $e$ of the divisor $E$ is contained in the diminished base locus of $M_Y$.
Indeed, for every point $e\in E$, we may find a curve $C_Y$ containing $e$ for which $M_Y\cdot C_Y<0$
and then $e\in C_Y \subset {\rm Bs}_{-}(M_Y/Z)$.
We conclude that every point $e\in E$ is contained
in ${\rm Bs}_{-}(M_Y/Z)$.
This means that $E\subset {\rm Bs}_{-}(M_Y/Z)$.
This contradicts Lemma~\ref{lem:dim-small}.
Thus, we have that $a>0$.
\end{proof}

\subsection{Generalized dlt modifications}
In this subsection, we prove some properties of
generalized dlt pairs and recall the concept of
$\qq$-factorial generalized dlt modification.
This is a straightforward generalization of $\qq$-factorial dlt modifications of log pairs (see, e.g.,~\cite{KK10}*{Theorem 3.1}).

\begin{definition}
\label{def:gdlt}
Let $(X/Z,B+M)$ be a generalized pair.
Let $\Gamma$ be the boundary divisor on $X$ obtained from $B$ 
by decreasing to one all its coefficients larger than one. 
We say that a projective birational morphism $\pi\colon Y\rightarrow X$ over $Z$ is a {\em generalized dlt modification} of $(X/Z,B+M)$ if it satisfies the following conditions:
\begin{enumerate}
    \item It only extracts divisors with non-positive generalized log discrepancy with respect to $(X,B+M)$, 
    \item the generalized pair $(Y/Z,\pi^{-1}_*\Gamma+E+M_Y)$ is generalized dlt, where $E$ is the reduced exceptional divisor, and 
    \item $K_Y+\pi^{-1}_*\Gamma+E+M_Y$ is nef over $X$.
\end{enumerate}
We say that the generalized dlt modification is $\qq$-factorial if $Y$ is a $\qq$-factorial variety.
\end{definition}

The following theorem is known as the existence of generalized dlt modifications.
In the case of generalized lc pairs, it follows by~\cite{BZ16}*{Lemma 4.5}.
In the glc case this statement is also proved in~\cite{Bir19}*{2.13.(3)}
and~\cite{HL18}*{Proposition 3.9}.
In the general case, it is proved by Filipazzi and Svaldi in~\cite{FS20}*{Theorem 2.9}.
The general statement was also included in Filipazzi's Ph.D. thesis~\cite{Fil19}*{Theorem 5 Chapter 4}.

\begin{theorem}\label{thm:gen-dlt-mod-4}
Let $(X/Z,B+M)$ be a generalized pair.
Then, $(X/Z,B+M)$ admits a 
$\qq$-factorial gdlt modification.
\end{theorem}

\begin{corollary}\label{cor:finiteness-glcc}
The number of generalized log canonical centers of a generalized log canonical pair is finite.
\end{corollary}

\begin{proof}
Let $(X/Z,B+M)$ be a generalized log canonical pair.
By Theorem~\ref{thm:gen-dlt-mod-4}, we can construct a $\qq$-factorial 
generalized dlt modification $(Y/Z,B_Y+M_Y)$ of $(X/Z,B+M)$.
By definition, the generalized log canonical centers of $(X/Z,B+M)$ are exactly the images on $X$ of the generalized log canonical centers of $(Y/Z,B_Y+M_Y)$.
On the other hand, the gdlt pair $(Y/Z,B_Y+M_Y)$ has finitely many generalized log canonical centers that correspond to the strata of $\lfloor B_Y\rfloor$. This proves the corollary.
\end{proof}

The following lemma is proved in~\cite{Bir19}*{Section 2.13.(2)}.

\begin{lemma}\label{lem:stay-gdlt}
Let $(X/Z,B+M)$ be a generalized dlt pair.
Assume we run a $(K_X+B+M)$-MMP over $Z$.
Then, every model in this minimal model program is generalized dlt.
\end{lemma}

\subsection{Generalized terminalization and small $\qq$-factorialization}

In this subsection, we prove that generalized klt pairs have a generalized $\qq$-factorial terminalization
and a generalized small $\qq$-factorialization. 

\begin{definition}
A {\em generalized $\qq$-factorial terminalization} of a generalized klt pair $(X/Z,B+M)$ is a $\qq$-factorial generalized terminal pair $(Y/Z,B_Y+M_Y)$ together with a projective birational morphism $\pi \colon Y \rightarrow X$ over $Z$
such that $K_Y+B_Y+M_Y=\pi^*(K_X+B+M)$. 
\end{definition}

\begin{definition}
A {\em generalized small $\qq$-factorialization} of a generalized klt pair $(X/Z,B+M)$ is a $\qq$-factorial generalized klt pair $(Y/Z,B_Y+M_Y)$
together with a small projective birational morphism $\pi \colon Y \rightarrow X$ over $Z$ such that $K_Y+B_Y+M_Y=\pi^*(K_X+B+M)$.
\end{definition}

The following proposition follows from~\cite{BZ16}*{Lemma 4.5}.

\begin{proposition}\label{prop:extraction-f}
Let $(X/Z,B+M)$ be a generalized klt pair.
Let $\mathcal{F}$ be a finite set of exceptional divisorial valuations over $X$ so that $a_E(X/Z,B+M)\in (0,1]$ for every $E\in \mathcal{F}$.
Then, there exists a projective birational morphism $\pi\colon Y\rightarrow X$ over $Z$ so that $Y$ is $\qq$-factorial and the exceptional divisors of $\pi$
correspond to the valuations in $\mathcal{F}$.
\end{proposition}

\begin{lemma}\label{lem:finite-klt} 
Let $(X/Z,B+M)$ be a generalized $\epsilon$-klt pair.
The set of exceptional divisorial valuations $E$ over $X$ with
$a_E(X/Z,B+M)\in (0,1+\epsilon)$ is finite.
\end{lemma}

\begin{proof}
Let $(X'/Z,B'+M')$ be a log resolution of $(X/Z,B+M)$ where $M$ descends.
Note that for any prime exceptional divisor $E$ over $X'$, 
we have that
\[
a_E(X'/Z,B'+M') = a_E(X'/Z,B').
\]
To conclude, we apply~\cite{KM98}*{Proposition 2.36} to the sub-pair $(X'/Z,B')$.
\end{proof}

\begin{lemma}\label{lem:finite-gter} 
Let $(X/Z,B+M)$ be a generalized terminal pair.
Let $b\in [0,1]$ so that the coefficients of $B$ are contained in $[0,b]$.
The set of exceptional divisorial valuations $E$ over $X$ with $a_E(X/Z,B+M)\in (0,2-b)$
is finite.
\end{lemma}

\begin{proof}
Let $(X'/Z,B'+M')$ be a log resolution of $(X/Z,B+M)$ where $M$ descends.
Note that for any prime exceptional divisor $E$ over $X'$, 
we have that
\[
a_E(X'/Z,B'+M') = a_E(X'/Z,B').
\]
To conclude, we apply~\cite{KM98}*{Proposition 2.36} to the sub-pair $(X'/Z,B')$.
\end{proof}

The following propositions are straightforward applications of Proposition~\ref{prop:extraction-f}.
These statements are also proved in~\cite{BZ16}*{Lemmata 4.5 and 4.6}. 

\begin{proposition}\label{existence-term}
A generalized klt pair $(X/Z,B+M)$ has a generalized $\qq$-factorial terminalization.
\end{proposition}

\begin{proof}
By Proposition~\ref{lem:finite-klt}, we know that $(X/Z,B+M)$ 
has finitely many exceptional divisorial valuations $E$ with 
$a_E(X/Z,B+M)\in (0,1]$.
We call $\mathcal{F}$ the finite set of such divisorial valuations.
Then, the statement follows from Proposition~\ref{prop:extraction-f}
applied to $\mathcal{F}$.
\end{proof}

\begin{proposition}\label{existence-small-Q}
A generalized klt pair $(X/Z,B+M)$ admits a generalized small $\qq$-factorialization.
\end{proposition}

\begin{proof}
This follows from Proposition~\ref{prop:extraction-f} applied to the set
$\mathcal{F}=\emptyset$.
\end{proof}

\section{Termination of gklt 3-fold quasi-flips}\label{S2}

In this section, we prove the termination of quasi-flips
for generalized klt $3$-folds with boundary coefficients in a DCC set.
The DCC condition on the coefficients is technical.
However, in the application of the lemmata of this section, there will be criteria to guarantee that the given MMP is under a DCC set.
We start by proving that any such sequence will terminate in codimension one, in the sense of the following definition.

\begin{definition}
We say that a sequence of birational transformations {\em terminates in codimension one}
if after finitely many birational transformations all maps are isomorphisms in codimension one.
In particular, a sequence of quasi-flips terminates in codimension one, if after finitely many quasi-flips,
all quasi-flips are weak.	
\end{definition}

\begin{lemma}~\label{tercod1}
Let $(X/Z,B+M)$ be a generalized log canonical pair.
Any sequence of strict ample klt quasi-flips for $(X/Z,B+M)$, under a DCC set, terminates in codimension one.
\end{lemma}

\begin{proof}
We claim that any prime divisor that is extracted in the sequence of quasi-flips has generalized log discrepancy at most one. 
Indeed, the generalized log discrepancy at the generic point of a prime component
of the boundary part with coefficient $0 < b_j < 1$ is $1-b_j$.
Therefore, by Proposition~\ref{monotonicity}, such prime component is a generalized non-terminal valuation over all previous models in the sequence of quasi-flips. 
Note that the exceptional locus of this MMP is contained in the gklt locus of $(X/Z,B+M)$.
Let $U\subset X$ be the open set on which $(X/Z,B+M)$ is generalized klt. 
We may apply Lemma~\ref{lem:finite-klt} to $(U/U,B+M)$, to conclude that there are finitely many 
generalized non-terminal valuations of 
$(X/Z,B+M)$ which are contained in the exceptional locus of this MMP.
We conclude that there are finitely many divisorial valuations that can be extracted as divisors in the sequence of strict ample quasi-flips.

Now, we prove that each of these finitely many divisorial valuations can be extracted at most finitely many times.
Let $\pi_i\colon X_i\dashrightarrow X_{i+1}$ be the $i$-th quasi-flip in the sequence.
For each $i$, we denote by $B_i$ the boundary part of the generalized pair on $X_i$
and by $M_i$ the trace of the b-divisor ${\bf M}$ on $X_i$.
Assume that a non-terminal valuation $E$ is extracted by the quasi-flip $\pi_{i}$ and the quasi-flip $\pi_{j}$ for $i <j$.
Denote by $E_{i+1}$ (resp. $E_{j+1}$) the center of $E$
on $X_{i+1}$ (resp. $X_{j+1})$.
Since $E$ is extracted by both $\pi_i$ and $\pi_j$, 
there exists a quasi-flip between $\pi_i$ and $\pi_j$
which contracts the center of $E$.
This means that it contracts the strict transform of $E_{i+1}$.
By Proposition~\ref{monotonicity}, we have that
\[
a_E(X_{i}/Z,B_{i}+M_{i})<a_E(X_{j}/Z,B_{j}+M_{j}).
\]
This means that
\begin{equation}\label{eq:strict-less-coeff}
{\rm coeff}_{E_{i+1}}(B_{i+1})
>
{\rm coeff}_{E_{j+1}}(B_{j+1}).
\end{equation} 
If there was an infinite sequence of quasi-flips 
that extracts $E$, 
then~\eqref{eq:strict-less-coeff} would imply that
there exists an infinite strictly decreasing sequence 
on the coefficients of the boundary.
Since the sequence of quasi-flips is under a DCC set, the coefficients of $B_i$ belong to a DCC set independent of $i$. 
We deduce that $E$ can only be extracted finitely many times.
Thus, every non-terminal valuation can be extracted only finitely many times.

We conclude that there are finitely many divisorial valuations that can be extracted in the sequence of quasi-flips,
and each of these can be extracted at most finitely many times.
Thus, after finitely many quasi-flips, all quasi-flips do not extract divisors.
By induction on the Picard rank of $X_i$ over $Z$, we conclude that after finitely many quasi-flips,
both the flipping and flipped contractions are isomorphisms in codimension one.
Indeed, each quasi-flip that contracts a divisor will drop the relative Picard rank by one.
Thus, after finitely many quasi-flips, 
the sequence of strict ample klt quasi-flips terminates in codimension one.
\end{proof}

Now, we turn to prove a lemma that will allow us to control
generalized log discrepancies in the interval $(0,1)$ of a generalized klt pair of dimension three. 
This lemma is motivated by~\cite{Sho96}*{Proposition 4.4}.
The conditions of the following lemma are hard to obtain in general settings. However, they are suitable to work with sequences of flips.

\begin{lemma}\label{lem:control-ci}
Let $N\in \zz_{\geq 0}$ and $\epsilon \in \rr_{>0}$.
There exists a constant $I:=I(N,\epsilon)$,
only depending on $N$ and $\epsilon$,
satisfying the following.
Let $(X/Z,B+M)$ be a $\qq$-factorial generalized 
pair of dimension $3$
such that there are at most $N$ divisorial valuations with
$a_E(X/Z,B+M)<1+\epsilon$.
Then, the Cartier index of every Weil divisor on $X$ divides $I$, i.e., for every Weil divisor $D$ on $X$, we have that $ID$ is Cartier.
\end{lemma}

\begin{proof}
Since $X$ is $\qq$-factorial, 
we know that $X$ is klt.
Indeed, by Lemma~\ref{lem:mono-ld}, we know that 
\[
0<a_E(X/Z,B+M)\leq a_E(X/Z),
\]
for every divisorial valuation $E$ over $X$.
We conclude that $X$ 
satisfies the following condition:
there are at most $N$ divisorial valuations 
with $a_E(X/Z)<1+\epsilon$.
By~\cite{Sho96}*{Lemma 4.4.1}, we conclude that
there exists a constant $I:=I(N,\epsilon)$, 
only depending on $N$ and $\epsilon$, 
such that the Cartier index of every Weil divisor on $X$ divides $I$.
\end{proof}

The following two statements allow us to control
certain log discrepancies
for generalized klt pairs. 
As usual, the divisor $M'$ denotes the trace of ${\bf M}$ on a model where it descends.
The condition $(3)$ in Lemmata~\ref{lem:blow-up-comp} and~\ref{lem:control-ld} is satisfied if ${\bf M}$ is NQC. 
Both lemmata use the finite set $\mathcal{R}$ explicitly in the statement.

\begin{lemma}\label{lem:blow-up-comp}
Let $\mathcal{R}\subset \mathbb{R}_{>0}$ be a finite set of positive real numbers.
Let $X$ be a projective variety over $Z$
and let $C\subset X$ be a curve.
Let $(X/Z,B+M)$ be a generalized klt $3$-fold satisfying the following conditions:
\begin{enumerate}
    \item the variety $X$ is smooth at the generic point of $C$,
    \item the coefficients of $B$ belong to $\mathcal{R}$, and
    \item we can write $M'\equiv_Z \sum_j \mu_j M'_j$, where each $M'_j$ is Cartier and
    nef over $Z$ and each $\mu_j\in \mathcal{R}$.
\end{enumerate}
Let $E$ be the exceptional divisor obtained by blowing up the curve $C$. Then, we can write 
\[
a_E(X/Z,B+M) = 
2- \sum_j b_jn_j -\sum_j \mu_j m_j,
\]
where $b_j,\mu_j\in \mathcal{R}$
and $n_j,m_j\in \mathbb{Z}_{\geq 0}$.
Furthermore, if a prime component $B_j$ of $B$ has coefficient $b_j$ and contains $C$, then the corresponding $n_j$ is positive. 
\end{lemma}

\begin{proof}
Let $\pi\colon Y\rightarrow X$ be the blow-up of the curve $C$.
By~\cite{KM98}*{Lemma 2.29}, we know that
$a_E(X/Z)=2$.
We write $B=\sum_j b_jB_j$, where $b_j\in \mathcal{R}$ 
and the $B_j$'s are prime Weil divisors.
Since we are computing the log discrepancy at $E$, we may assume that
each divisor $B_j$ contains the curve $C$.
Furthermore, for each $j$, we have that $B_j$ is Cartier at the generic point of $C$.
Thus, we have that 
\[
{\rm coeff}_E(\pi^*B)= \sum_j b_jn_j,
\]
for certain positive integers $n_j$.
Thus, we can write
\[
a_E(X/Z,B)= 2-\sum_j b_jn_j.
\]
Finally, we can write 
$M\equiv_Z \sum_j \mu_j M_j$, where each $M_j$ is the push-forward of a nef Cartier divisor on a higher model.
In particular, each $M_j$ is pseudo-effective and Weil.
Thus, for each $j$, the divisor $M_j$ is Cartier at the generic point of $C$.
By the negativity lemma~\cite{KM98}*{Proposition 3.39}, we can write
\[
\pi^*M \equiv_Z \sum_j \mu_j M_{j,Y} + \left(\sum_j \mu_j m_j\right)E.
\]
In the above formula, the numbers $m_j$ are nonnegative integers.
Thus, we conclude that 
\[
a_E(X/Z,B+M) = 
2- \sum_j b_j n_j -\sum_j \mu_j m_j,
\]
as claimed.
The last statement follows from the fact that $\pi^*B_j$ contains $E$ in its support if $B_j$ contains $C$.
\end{proof}

\begin{lemma}\label{lem:control-ld}
Let $N\in \mathbb{Z}_{\geq 0}$ 
and $\epsilon \in \mathbb{R}_{>0}$.
Let $\mathcal{R}\subset \mathbb{R}_{>0}$ be a finite set of positive real numbers. 
Let $A(\mathcal{R},N,\epsilon)$ be the set of generalized log discrepancies $\leq 1$ for generalized $\epsilon$-klt pairs $(X/Z,B+M)$ of dimension $3$, satisfying the following conditions:
\begin{enumerate}
\item there are at most $N$ divisorial valuations with 
$a_E(X/Z,B+M)<1+\epsilon$, 
\item the coefficients of $B$ belong to $\mathcal{R}$, and
\item we can write $M' \equiv_Z \sum \mu_j M'_j$, where each $M'_j$ is Cartier and nef over $Z$ and $\mu_j\in \mathcal{R}$.
\end{enumerate}
Then, the set $A(\mathcal{R},N,\epsilon)$ is finite.
\end{lemma}

We note that the statement of Lemma~\ref{lem:control-ld}
is immediate from Lemma~\ref{lem:control-ci} in the case of $\qq$-divisors. 
The difficulty of the statement lies in the $\rr$-coefficients case.

\begin{proof}
Passing to a small $\qq$-factorialization, 
we may assume that each $X$ is $\qq$-factorial.
Let $E$ be a divisorial valuation over $X$ satisfying
$a_E(X/Z,B+M)\in (0,1)$.
By Proposition~\ref{prop:extraction-f}, we can find a 
projective birational morphism $\pi \colon Y\rightarrow X$ over $Z$
that extracts $E$ and its exceptional locus contains a unique divisor.
Since $Y$ and $X$ are $\qq$-factorial, we have that $\rho(Y/X)=1$.
Furthermore, $-E$ is ample over $X$ so the exceptional locus is divisorial.
We can write
\[
\pi^*(K_X+B+M)=
K_Y+cE+B_Y+M_Y,
\]
where $c<1-\epsilon$.
Here, as usual, $B_Y$ is the strict transform of $B$ on $Y$
and $M_Y$ is the push-forward of $M'$ to $Y$.
We aim to show that $c$ belongs to a finite set
that only depends on $\mathcal{R},N$, and $\epsilon$.
By Lemma~\ref{lem:control-ci}, we know that the Cartier index of every Weil divisor on $Y$ divides $I$, where $I$ only depends on $N$ and $\epsilon$.

Now, we proceed to argue that $c$ belongs to a finite set.
Since $\pi\colon Y\rightarrow X$ is an extremal contraction, we have that
$E\cdot C <0$ for any curve $C$ that is contracted by $\pi$.
By~\cite{Sho94}*{Theorem and Remark (3)}, 
we know there exists a curve $C_0 \subsetneq E$,
movable on $E$, with
\[
0>(K_Y+E)\cdot C_0 \geq -3.
\]
The inequality $0>(K_Y+E)\cdot C_0$ holds 
as $E\cdot C_0<0$ and $(B_Y+M_Y)\cdot C_0\geq 0$,
being $C_0$ movable on $E$.
Since the curve $C_0$ is movable on $E$, 
we have that $B_Y\cdot C_0\geq 0$.
By Lemma~\ref{dim:movable-cod-1}, 
we know that
$M_Y\cdot C_0\geq 0$.
We conclude that $(K_Y+cE)\cdot C_0 \leq 0$.
Thus, we deduce that 
\[
(K_Y+E)\cdot C_0 \geq (K_Y+cE)\cdot C_0-3,
\]
from where we get that $0>E\cdot C_0\geq -3\epsilon^{-1}$.
Recall that $E\cdot C_0 \in \zz[I^{-1}]$, where $I$ only depends on $N$ and $\epsilon$. 
We conclude that the intersection
$E\cdot C_0$ can only take finitely many values in $(0,-3\epsilon^{-1}]$.
On the other hand, we have that $3\epsilon^{-1}\geq K_Y\cdot C_0 \geq -3$,
so $K_Y\cdot C_0$ can only take finitely many values.

We claim that $(B_Y+M_Y)\cdot C_0$ can only take finitely many values.
Observe that $(B_Y+M_Y)\cdot C_0\geq 0$.
On the other hand, we can write
\[
B_Y+M_Y= \sum_j b_j B_{Y,j} + \sum_j \mu_j M_{Y,j},
\]
where $b_j,\mu_j\in \mathcal{R}$ and the $B_{Y,j}$'s
and $M_{Y,j}$'s are Weil divisors.
Hence, we have that
\[
(B_Y+M_Y)\cdot C_0 =
\left(\sum_j b_j B_{Y,j} + \sum_j \mu_j M_{Y,j} \right)\cdot C_0 =
\sum_j b_j\frac{n_j}{I} +\sum_j \mu_j \frac{m_j}{I},
\]
where the $n_j$'s and $m_j$'s are nonnegative integers.
Since 
\[
0\leq 
\sum_j b_j\frac{n_j}{I} +\sum_j \mu_j \frac{m_j}{I}  \leq
-(K_Y+E)\cdot C_0 \leq 3,
\]
we conclude that $(B_Y+M_Y)\cdot C_0$ can only take finitely many values.

Finally, since 
\[
(K_Y+cE+B_Y+M_Y)\cdot C_0 =0,
\]
then we can write
\begin{equation}\label{eq:finiteness-intersection}
c = \frac{ K_Y\cdot C_0 + (B_Y+M_Y)\cdot C_0}{-E\cdot C_0}.
\end{equation}
Note that in~\eqref{eq:finiteness-intersection} both the numerator and denominator
belong to a finite set only depending on $\mathcal{R},N$, and $\epsilon$.
 We conclude that $1-c=a_E(X/Z,B+M)$ belongs to a finite set
only depending on $\mathcal{R},N$ and $\epsilon$.
\end{proof}

\begin{proposition}\label{ter3folds}
Let $(X/Z,B+M)$ be a generalized NQC klt pair of dimension $3$.
Then, any sequence of strict ample quasi-flips
for $(X/Z,B+M)$,
under a DCC set, terminates.
\end{proposition}

\begin{proof}
Since $(X/Z,B+M)$ is NQC, the b-nef divisor ${\bf M}$ is NQC.
Let $X'\rightarrow X$ be a model where ${\bf M}$ descends
and $M'$ be the trace of ${\bf M}$ on $X'$.
By the NQC assumption, 
we can write
$M'\equiv_Z \sum_j \mu_j M'_j$
where each
$M'_j$ is Cartier and nef over $Z$
and each 
$\mu_j$ is a nonnegative real number.

We prove this statement in several steps.
We have a sequence of strict ample quasi-flips as follows:
\begin{equation}\label{eq:sequence-strict-ample}
 \xymatrix{
  (X/Z,B+M)\ar@{-->}^-{\pi_{1}}[r] & (X_1/Z,B_1+M_1)\ar@{-->}^-{\pi_{2}}[r] & \dots \ar@{-->}^-{\pi_{i}}[r] & (X_i/Z,B_i+M_i)\ar@{-->}^-{\pi_{i+1}}[r] & \dots   \\ 
 }
\end{equation}
In the previous sequence, 
$M_i$ is the trace of ${\bf M}$ on $X_i$, i.e., 
the equality $M_i={\bf M}_{X_i}$ holds.
The divisor $B_i$ is the boundary part of the generalized pair $(X_i/Z,B_i+M_i)$.
Recall that, by assumption, the coefficients of $B_i$ belong to a fixed DCC set
which is independent of $i$. 
As usual, the flipping contraction of $\pi_i$
will be denoted by $\phi_i \colon X_i \rightarrow W_i$ and the flipped contraction will
be denoted by $\phi^+_i\colon X_{i+1}\rightarrow W_i$.
Here, we are setting $(X_0/Z,B_0+M_0)=(X/Z,B+M)$.\\

\noindent\textit{Step 1:} We reduce to the case in which each $\pi_i$ is small and ${\pi_{i+1}}_* B_i =B_{i+1}$.\\

 By Lemma~\ref{tercod1}, we conclude that
 a sequence as in~\eqref{eq:sequence-strict-ample} terminates in codimension one.
 Thus, after finitely many strict ample quasi-flips, all the quasi-flips will be
 strict ample weak quasi-flips.
 Truncating the sequence~\eqref{eq:sequence-strict-ample}, we may assume that each $\pi_i$ is a strict ample weak quasi-flips.
 By definition of quasi-flips, we have that
 \[
 {\pi_{i+1}}_*B_i\geq B_{i+1}.
 \]
 Let $P\subset X$ be a prime divisor.
 Denote by $P_i$ the strict transform of $P$ on $X_i$.
 The sequence 
 \[
 {\rm coeff}_{P_1}(B_1) \geq
 {\rm coeff}_{P_2}(B_2) \geq \dots \geq
 {\rm coeff}_{P_i}(B_i) \geq \dots 
 \]
 is decreasing, and it belongs to a DCC set.
 Hence, it must stabilize.
 Applying the same argument to each prime divisor on the support of $B$,
 we conclude that, after finitely many
 strict ample weak quasi-flips, 
 the equality ${\pi_{i+1}}_* B_i=B_{i+1}$ holds.
 Truncating the sequence~\eqref{eq:sequence-strict-ample} again, we may assume that each $\pi_i$ is small 
 and ${\pi_{i+1}}_* B_i=B_{i+1}$ holds for each $i\geq 0$.\\

\noindent\textit{Step 2:}  We reduce to the case of $\qq$-factorial flips.\\
 
Consider a strict ample weak klt quasi-flip for a generalized klt pair 
\[
\pi_1 \colon (X/Z,B+M)\dashrightarrow (X_1/Z,B_1+M_1)
\]
with flipping contraction $\phi_1 \colon X \rightarrow W$.
By Lemma~\ref{existence-small-Q}, we can take a small $\qq$-factorialization $(X'/Z,B'+M')$ of $(X/Z,B+M)$.
By ~\cite{BZ16}*{Theorem 4.4}, we can run a relative minimal model program for $(X'/Z,B'+M')$ over $W$
to produce a minimal model $(X'_1/Z,B'_1+M'_1)$ over $W$.
Here, $B'_1$ is the strict transform of $B'$ on $X'_1$.
By the uniqueness of ample models, 
the generalized pair $(X_1/Z,B_1+M_1)$ is the ample model of $(X'_1/Z,B'_1+M'_1)$ over $W$.
In particular, the minimal model
$(X'_1/Z,B'_1+M'_1)$ is a small 
$\qq$-factorialization of $(X_1/Z,B_1+M_1)$.
By construction, the small birational map $(X'/Z,B'+M')\dashrightarrow (X_1/Z,B_1+M'_1)$
decomposes as a sequence of flips over $Z$ for a $\qq$-factorial generalized klt pair over $Z$.
Proceeding inductively with the other strict ample weak
quasi-flips $\pi_i$,
we obtain a sequence of $\qq$-factorial flips.
We denote by $(X_i'/Z,B_i'+M_i')$ the induced small $\qq$-factorialization
of $(X_i,B_i+M_i)$. 
We argue that if the sequence of $\qq$-factorial flips stabilizes, 
then so does the starting sequence. 
Indeed, if the sequence of $\qq$-factorial flips stabilize, then 
for $i\gg 0$
we have $a_E(X_i',B_i'+M_i')=a_E(X'_{i+1},B_{i+1}'+M_{i+1}')$
for every $E$. 
As $(X'_i,B_i'+M_i')\rightarrow (X_i,B_i+M_i)$ is crepant, then
for $i\gg 0$, we have 
$a_E(X_i,B_i+M_i)=a_E(X'_{i+1},B_{i+1}+M_{i+1})$ for every $E$.
Proposition~\ref{monotonicity}, then implies that each $\pi_i$ with $i\gg0$ must be the identity, contradicting the strict condition of the quasi-flip $\pi_i$ (see also Remark~\ref{rem}).
Thus, we may assume that the sequence~\eqref{eq:sequence-strict-ample} is a sequence of $\qq$-factorial flips
and ${\pi_{i+1}}_* B_i=B_{i+1}$ for each $i$.

After these reductions, we may assume that the coefficients of all the boundary divisors $B_i$'s belong to a finite set of positive real numbers.
Thus, there exists a finite set of positive real numbers $\mathcal{R}$, that contains the $\mu_j$'s and all the coefficients of the $B_i$'s.\\

\noindent\textit{Step 3:} We reduce to the case of generalized terminal flips.\\

We are assuming that each $(X_i/Z,B_i+M_i)$ is generalized klt. 
We may let $(X/Z,B+M)$ be the first generalized klt model in the truncated sequence.
Let $\epsilon$ be a positive real number such that
$(X/Z,B+M)$ is generalized $\epsilon$-klt.
By Lemma~\ref{lem:finite-klt}, we know that there is a finite number $N$ of divisorial valuations over $(X/Z,B+M)$ with log discrepancy in the interval $(\epsilon,1+\epsilon)$. 
By the monotonicity Lemma~\ref{monotonicity}, we conclude that for every $i$, the generalized klt pair $(X_i/Z,B_i+M_i)$ has at most $N$ divisorial valuations with log discrepancy in the interval
$(\epsilon,1+\epsilon)$. 
By Lemma~\ref{lem:control-ci}, we know there exists a constant $I$, only depending on $N$ and $\epsilon$, such that the Cartier index of any Weil divisor on any $X_i$ divides $I$.
Furthermore, we conclude that for any $i$ and $E$ with $a_E(X_i/Z,B_i+M_i)\in (0,1]$, we have that
\[
a_E(X_i/Z,B_i+M_i) \in A(\mathcal{R},N,\epsilon).
\]
By Lemma~\ref{lem:control-ld}, the set $A(\mathcal{R},N,\epsilon)$ is a finite set
that only depends on $\mathcal{R},N$ and $\epsilon$.

We claim that after finitely many flips, no flipping locus will contain the center of a generalized non-terminal valuation.
By Proposition~\ref{monotonicity}, each $(X_i/Z,B_i+M_i)$ has finitely many generalized non-terminal valuations, which are a subset of the generalized non-terminal valuations of $(X/Z,B+M)$.
Let $E$ be a generalized non-terminal valuation of $(X/Z,B+M)$.
Note that for each $i$,
if $a_E(X_i/Z,B_i+M_i)\in (0,1]$, then 
$a_E(X_i/Z,B_i+M_i)\in A(\mathcal{R},N,\epsilon)$.
On the other hand, 
if the center of $E$ on $X_i$ is contained in the flipping locus of $\pi_{i+1}$, then we have 
\[
a_E(X_i/Z,B_i+M_i)<a_E(X_{i+1}/Z,B_{i+1},M_{i+1}).
\] 
Since the set $A(\mathcal{R},N,\epsilon)$ is finite, then the conditions:
\begin{enumerate}
    \item $a_E(X_i/Z,B_i+M_i)\in (0,1]$ and,
    \item $c_{X_i}(E)$ is contained in the flipping locus of $\pi_{i+1}$,
\end{enumerate}
can happen only finitely many times.
We conclude that, after finitely many flips,
either no flipping locus contains the center of $E$ in its flipping locus or $E$ is not a generalized non-terminal valuation.
In either case, we conclude that after finitely many flips, no flip will contain the center of a generalized non-terminal valuation. 
In particular, the terminalizations of the generalized pairs in the sequence of flips
are isomorphic in codimension $1$.
By Proposition~\ref{existence-term},
proceeding as in Step 2, we can take a $\qq$-factorial terminalization of each $(X_i,B_i+M_i)$ to reduce to the case
of generalized terminal $3$-fold flips.
Thus, truncating the sequence~\eqref{eq:sequence-strict-ample} again, we may assume that we have a sequence of flips
for generalized terminal $3$-folds.\\

\noindent\textit{Step 4:} We reduce to the case in which $B=0$.\\ 

Let $b$ be the maximal real number among the coefficients of $B$.
By Lemma~\ref{lem:finite-gter}, we know that 
each generalized terminal pair $(X_i/Z,B_i+M_i)$ has finitely many exceptional divisorial valuations with log discrepancy in the interval $(0,2-b)$.
We claim that, after finitely many flips, every flipping locus is disjoint from the components of $B$ whose coefficient is $b$.
Let $D_i\subset B_i$ be a component 
with ${\rm coeff}_{D_i}(B_i)=b$.
First, assume that the flipping curves are $D_i$-negative.
Then, the flipping locus is contained in $D_i$.
By Lemma~\ref{lem:mono-ld}, we know that $X_i$ has terminal singularities. In particular, $X_i$ is smooth at the generic point of each irreducible component of the flipping locus.
Let $Y_i\rightarrow X_i$ be the blow-up of an irreducible component of the flipping locus and let $E$ be the unique exceptional divisor.
By Lemma~\ref{lem:blow-up-comp}, we have that
\[
a_E(X_i/Z,B_i+M_i)=2-\sum b_j\frac{n_j}{I} 
- \sum \mu_j\frac{m_j}{I},
\]
where the $n_j$'s and $m_j$'s are nonnegative integers, and at least one $n_j$ with $b_j=b$ is positive.
Hence, we conclude that $a_E(X_i/Z,B_i+M_i)<2-b$.
Observe that the same argument applies if the flipping curves are $D_i$-trivial and are not disjoint from $D_i$.
Indeed, in such case, the flipping curves are contained in $D_i$ as well.
Analogously, if the flipping curves are $D_i$-positive, then the flipped locus is contained in the strict transform of $D_i$. 
In particular, the flipped locus is contained in a component of $B_{i+1}$ with coefficient $b$.
The same argument yields that 
$a_E(X_{i+1}/Z,B_{i+1}+M_{i+1})<2-b$.
Note that
\[
\mathcal{S}(\mathcal{R},I):=
\left\{ 
2-\sum b_j\frac{n_j}{I} 
- \sum \mu_j\frac{m_j}{I} \mid 
b_j,\mu_j\in \mathcal{R},
n_j,m_j\in \mathbb{Z}_{\geq0} 
\right\} \cap \mathbb{R}_{>0}
\subset \mathbb{R}_{>0}
\] 
is a finite set that only depends on $\mathcal{R}$ and $I$.
Thus, every time that the flipping locus is contained in a component $B_i$ with coefficient $b$, we can find a divisorial valuation whose center is contained in the flipping locus whose log discrepancy belongs to $\mathcal{S}(\mathcal{R},I)\cap (0,2-b)$.
Analogously, every time that the flipped locus 
is contained in a component of $B_{i+1}$ with coefficient $b$, we can find a divisorial valuation whose center is contained in the flipped locus and whose log discrepancy belongs to the finite set $\mathcal{S}(\mathcal{R},I)\cap (0,2-b)$.
Since there are only finitely many exceptional divisorial valuations with generalized log discrepancy in the interval $(0,2-b)$ in this sequence
and $\mathcal{S}(\mathcal{R},I)$ is finite,
we conclude that this can only happen finitely many times.
Indeed, the generalized log discrepancy 
strictly increases if the center of the valuation
is contained in the flipping locus (see Lemma~\ref{monotonicity}).
Thus, after finitely many flips, all flips are disjoint from the components of $B_i$ with coefficient $b$.
Hence, we can truncate the sequence~\eqref{eq:sequence-strict-ample}, so that all flips are disjoint from such components.
Since each $X_i$ is $\qq$-factorial, we can reduce the coefficients of such components to zero.
Indeed, the original MMP remains an MMP 
for the pair $(X_i/Z,B_i-bD_i+M_i)$
as all the flipping curves are $D_i$-trivial.
Then, we can pick the new largest coefficient of the $B_i$'s, and proceed inductively.
We deduce that, after finitely many flips, the flipping locus (and hence the flipped locus) are disjoint from the boundary divisor.
Thus, we can reduce the boundary divisor to zero.
Truncating the sequence~\eqref{eq:sequence-strict-ample} again, we may assume that we have a sequence of flips for $\qq$-factorial generalized terminal pairs without boundary.
\\ 

\noindent\textit{Step 5:} We reduce to the case in which $M=0$.\\ 

In the previous step,
we reduced to the case of flips for generalized $\qq$-factorial terminal pairs without boundary: 
\[
 \xymatrix{
  (X/Z,M)\ar@{-->}^-{\pi_{1}}[r] & (X_1/Z,M_1)\ar@{-->}^-{\pi_{2}}[r] & \dots \ar@{-->}^-{\pi_{i}}[r] & (X_i/Z,M_i)\ar@{-->}^-{\pi_{i+1}}[r] & \dots   \\ 
 }
 \]
By Lemma~\ref{lem:finite-gter}, we know that for each $i$, the pairs $(X_i/Z,M_i)$ have finitely many log discrepancies in the interval $(0,2)$.
Assume that the flip
$\pi_1$ is $M$-negative.
Then, the flipping locus is contained in the diminished base locus of $M$ relatively over $Z$.
Consider an irreducible curve $C$ in the flipping locus. Let $E$ be the exceptional divisor of the blow-up of $C$.
Note that, by Lemma~\ref{lem:blow-up-comp}, we can write
\[
a_E(X/Z,M)=2-\sum \mu_j m_j.
\]
On the other hand,
by Lemma~\ref{lem:blow-up-M},
we conclude that 
$a_E(X/Z,M)<a_E(X/Z)=2$, i.e., 
at least one $m_j>0$.
Thus, we conclude that the divisorial valuation $E$ belongs to the finite set of divisorial valuations with log discrepancy in $(0,2)$.
On the other hand, the set
\[
\mathcal{S}'(\mathcal{R}):=\left\{ 
2-\sum \mu_j m_j \mid \mu_j \in \mathcal{R},
m_j\in \mathbb{Z}_{\geq 0} \right\} 
\cap \mathbb{R}_{\geq 0} \subset \mathbb{R}_{\geq 0}
\]
is a finite set that only depends on $\mathcal{R}$.
If the flip $\pi_1$ is $M$-positive, then the same argument applies to the flipped contraction.
In any case, we can find a divisorial valuation with log discrepancy in the interval $(0,2)\cap \mathcal{S}'(\mathcal{R})$.
Since the set of divisors over $(X/Z,M)$ with log discrepancy in $(0,2)$ is finite and the set $\mathcal{S}'(\mathcal{R})$ is finite, we conclude that this can only happen finitely many times.
Thus, after finitely many flips, all the flips are $M$-trivial, hence $K_X$-flips.
Thus, truncating the sequence~\eqref{eq:sequence-strict-ample} again, we are left with a sequence of terminal threefold flips.
This sequence terminates by~\cite{KM98}*{Theorem 6.17}.
We conclude that the initial sequence terminates.
\end{proof}

\section{\texorpdfstring{Special termination for $\qq$-factorial gdlt 4-fold flips}{Special termination for Q-factorial gdlt 4-fold flips}}\label{S3}

In this section, we prove the special termination of $\qq$-factorial
gdlt $4$-fold flips.
This means that in any sequence of $\qq$-factorial gdlt $4$-fold flips, the flips are eventually disjoint from the generalized log canonical centers.
Special termination for lc pairs has been studied by Shokurov and Fujino
(see, e.g.,~\cite{Sho04}*{\S 4}
and~\cite{Fuj07}).
We recall the divisorial generalized adjunction proved in~\cite{BZ16}.
We phrase the following proposition in terms of $4$-folds since this is the case in which we will apply it.
The following is also proved in~\cite{HL18}*{2.8}.

\begin{proposition}{\rm (cf.~\cite{BZ16}*{Proposition 4.9})}\label{adj}
Let $\Lambda$ be a DCC set of nonnegative real numbers.
There exists a DCC set $\Omega$ of nonnegative real numbers, only depending on $\Lambda$, satisfying the following.
Let $(X/Z,B+M)$ be a  gdlt $4$-fold and $S$ be a prime component of $\lfloor B\rfloor$.
Let $X'\rightarrow X$ be a model where ${\bf M}$ descends
and let $M'$ be the trace of ${\bf M}$ on $X'$.
Assume that the following conditions are satisfied:
\begin{enumerate}
\item we can write $M'\equiv_Z \sum_j \lambda_j M'_j$, where each
$M'_j$ is Cartier and nef over $Z$
and $\lambda_j\in \Lambda$, and
\item the coefficients of $B$ belong to $\Lambda$.
\end{enumerate}
Then, we can write
\begin{equation}\label{eq:adj}
(K_X+B+M)|_S = K_S+B_S+M_S,
\end{equation}
where $(S/Z,B_S+M_S)$ is a generalized dlt $3$-fold satisfying:
\begin{enumerate}
\item we can write $M'_S \equiv_Z \sum_j \omega_j M'_{S,j}$, where each $M'_{S,j}$ is Cartier and nef over $Z$ and $\omega_j \in \Omega$, and 
\item the coefficients of $B_S$ belong to $\Omega$.
\end{enumerate}
\end{proposition}

\begin{proof}
The adjunction formula~\eqref{eq:adj} is defined in~\cite{BZ16}*{Definition 4.7}.
All the statements of the proposition are proved
in~\cite{BZ16}*{Proposition 4.9}, 
except the generalized dlt-ness of $(S/Z,B_S+M_S)$.
Note that $M'_S$ is defined by the pull-back of $M'$ to the strict transform of $S$ on a higher model of $X$. In particular, we can write
$M'_S \equiv_Z \sum_j \lambda_j M'_{S,j}$, where
$\lambda_j\in \Lambda$.
On the other hand, by~\cite{BZ16}*{Proposition 4.9}, we know that the coefficients of $B_S$ belong to a DCC set $\Omega$ that only depends on $\Lambda$.
In~\cite{BZ16}*{Definition 4.7}, it is proved that
$(S/Z,B_S+M_S)$ is glc.
It suffices to prove that it is gdlt.
Let $U\subset X$ be as in the definition of gdlt.
Observe that $U_S:=U\cap S$ is a non-empty open subset of $S$.
Note that $U_S$ is smooth and $B_S|_{U_S}$ has simple normal crossing support.
Furthermore, every generalized log canonical center of $(X/Z,B+M)$ intersects $U$ non-trivially and is given by strata of $\lfloor B \rfloor$.
By~\cite{Fil20}*{Theorem 1.6}, we deduce that every log canonical center of $(S/Z,B_S+M_S)$ intersects $U\cap S$ non-trivially and is given by strata of $\lfloor B_S\rfloor$.
We conclude that the pair $(S/Z,B_S+M_S)$ is generalized dlt.
\end{proof}

\begin{proposition}\label{qflipslcp}
Let $(X/Z,B+M)$ be a NQC generalized dlt $4$-fold.
Let $V$ be a log canonical center of $(X/Z,B+M)$.
Consider a sequence of $(K_X+B+M)$-flips
\begin{equation}\nonumber
 \xymatrix{
  (X/Z,B+M)\ar@{-->}^-{\pi_{1}}[r] & (X_1/Z,B_1+M_1)\ar@{-->}^-{\pi_{2}}[r] & (X_2/Z,B_2+M_2)\ar@{-->}^-{\pi_3}[r] & \dots \ar@{-->}^-{\pi_{i}}[r] & (X_i/Z,B_i+M_i)\ar@{-->}^-{\pi_{i+1}}[r] & \dots   \\ 
 }
\end{equation}
that does not contain $V$ in any flipping locus. 
Then it induces a sequence of birational transformations 
\begin{equation}\nonumber
 \xymatrix{
  (V/Z,B_V+M_V)\ar@{-->}^-{\pi_{V,1}}[r] & (V_1/Z,B_{V_1}+M_{V_1})\ar@{-->}^-{\pi_{V,2}}[r] & (V_2/Z,B_{V_2}+M_{V_2})\ar@{-->}^-{\pi_{V,3}}[r] & \dots \ar@{-->}^-{\pi_{V,i}}[r] & (V_i/Z,B_{V_i}+M_{V_i})\ar@{-->}^-{\pi_{V,i+1}}[r] &\\ 
 }
\end{equation}
where $(V/Z,B_V+M_V)$ is a generalized dlt pair.
Each birational map $\pi_{V,i+1}$ is either a strict ample $(K_{V_i}+B_{V_i}+M_{V_i})$-quasi-flip 
or the identity. In the latter case, the flipping locus of $\pi_{i+1} \colon X_{i}\dashrightarrow X_{i+1}$ is disjoint from $V$.
Furthermore, the sequence $\pi_{V,i}$ is under a DCC set.
\end{proposition}

\begin{proof}
Note that $V$ is the intersection of exactly $\codim_X(V)$ components of $\lfloor B\rfloor$.
For each generalized dlt pair $(X_i/Z,B_i+M_i)$, we can apply Proposition~\ref{adj} $\codim_X(V)$ times to obtain a generalized dlt pair $(V_i/Z,B_{V,i}+M_{V,i})$.
In particular, the sequence of birational maps $\pi_{V,i}$ is under a DCC set.
This proves the last assertion.
It is clear that $V_{i+1}$ is the strict transform of $V_i$ with respect to $\pi_{i+1}$.
This holds because $\pi_i$ is an isomorphism at the generic point of $V_i$.

By induction, it suffices to prove that $\pi_1$ induces a strict ample $(K_V+B_V+M_V)$-quasi-flip if its
flipping locus intersects $V$ non-trivially.
Let $\phi\colon X\rightarrow W$ and $\phi^+\colon X^+\rightarrow W$ be the flipping contraction and flipped contraction of $\pi_1$, respectively.
Since $(V/Z,B+M)$ 
and $(V_1/Z,B_{V_1}+M_{V_1})$ are both gldt, then the first condition of Definition~\ref{qfgen} holds. 
Let $C\subset V_1$ be a curve that is being contracted by the flipping  contraction. 
By the adjunction formula~\eqref{eq:adj}, we have that
\[
(K_{V}+B_{V}+M_{V})\cdot C = (K_{X}+B+M)|_{V} \cdot C < 0.
\]
We conclude that the $\rr$-Cartier $\rr$-divisor $K_{V}+B_{V}+M_{V}$ is anti-ample over $W$. 
Analogously, we can check that $K_{V_1}+B_{V_1}+M_{V_1}$ is ample over $W$, so the second condition of Definition~\ref{qfgen} holds.
We claim that the projective birational map $V \dashrightarrow V_1$ is a
$(K_{V}+B_{V}+M_{V})$-negative birational map.
This means that we can find a common resolution
$p \colon V' \rightarrow V$ and $q\colon V'\rightarrow V_1$, where the following inequality holds
\[
 p^*( K_{V}+B_{V}+M_{V} ) - q^*( K_{V_1}+B_{V_1}+M_{V_1}) \geq 0.
\]
Indeed, we can take $V'$ to be the strict transform of $V$ 
on a log resolution of the flip $X\dashrightarrow X_1$ that is an isomorphism at the generic point of $V$.
The existence of such log resolution follows from the generalized dlt condition of the pair $(X/Z,B+M)$.
In particular,
both $(X/Z,B+M)$ and $(X_1/Z,B_1+M_1)$ are log smooth at the generic point of $V$ and $V_1$, respectively.
Let $p_X$ and $q_X$ be the projections to $X$ and $X_1$ from this common resolution.
We can write
\[
p^*( K_{V}+B+M ) - q^*( K_{V_1}+B_{V_1}+M_{V_1}) = 
(p_X^*(K_X+B+M)-q_X^*(K_{X_1}+B_1+M_1))|_{V'}= E \geq 0.
\]
Indeed, since $V$ is not contained in the flipping locus, the support of the effective divisor $p^*( K_{V}+B+M ) - q^*( K_{V_1}+B_{V_1}+M_{V_1})$ does not contain $V'$.
Therefore, we have
the following inequality 
\[
\psi^+_* B_{V_1} = \psi^+_*(B_V - p_*E)  \leq  \psi_* B_V,
\]
where $\psi$ and $\psi^+$ are the flipping and flipped contraction of $\pi_{V,1}$.
This gives us the third condition of Definition~\ref{qfgen}.

Finally, it suffices to check that the flipping locus of $\pi_{i}$ is disjoint from $V$ whenever $\pi_{V,i}$ is the identity.
Analogously, it suffices to show that whenever the flipping locus of $\pi_i$ intersects $V$, the map $\pi_{V,i}$ is not the identity. We assume that the flipping locus of $\pi_1$ intersects $V$. 
Observe that the flip $\pi_1 \colon X \dashrightarrow X_1$ is ample.
By Proposition~\ref{monotonicity}, the coefficient of 
\begin{equation}\label{ambientdifference}
p_X^*(K_X+B+M)-q_X^*(K_{X_1}+B_1+M_1)
\end{equation}
at any $p_X$-exceptional or $q_X$-exceptional prime divisor with center on the flipping or flipped locus is positive.
It suffices to show that there exists a component $E_i$ in the support of~\eqref{ambientdifference} that intersects $V$ non-trivially.
Indeed, by the negativity lemma~\cite{KM98}*{Proposition 3.39.(2)} a fiber of $p$ is either disjoint from the support of~\eqref{ambientdifference} or is contained in its support. Therefore, a fiber of $p$ over the intersection of $V$ and the flipping locus of $\pi_1 \colon X \dashrightarrow X_1$
must be contained in the support of the divisor in~\eqref{ambientdifference}.
We conclude that there exists a prime divisor $E_i$ in the support of~\eqref{ambientdifference} that intersects $V$ non-trivially.
We deduce that the divisor $E$ is non-trivial. This implies that $\pi_{V,1}\colon V\dashrightarrow V_1$ is not the identity. 
\end{proof}

\begin{proposition}[{\rm Special termination}]\label{special-term}
Let $(X/Z,B+M)$ be a NQC generalized dlt $4$-fold.
Then, a sequence of flips for $(X/Z,B+M)$ over $Z$ is eventually disjoint from the generalized log canonical centers of $(X/Z,B+M)$.
\end{proposition}

\begin{proof}
By Corollary~\ref{cor:finiteness-glcc}, we know that $(X/Z,B+M)$ has finitely many generalized log canonical centers.
Consider a sequence of flips for $(X/Z,B+M)$ as follows
\begin{equation}\nonumber
 \xymatrix{
  (X/Z,B+M)\ar@{-->}^-{\pi_{1}}[r] & (X_1/Z,B_1+M_1)\ar@{-->}^-{\pi_{2}}[r] & (X_2/Z,B_2+M_2)\ar@{-->}^-{\pi_3}[r] & \cdots \ar@{-->}^-{\pi_{i}}[r] & (X_i/Z,B_i+M_i)\ar@{-->}^-{\pi_{i+1}}[r] & \cdots    \\
 }
\end{equation}
We may assume that the number of generalized log canonical centers stabilizes. 
Hence, by Proposition~\ref{monotonicity}, no generalized log canonical center is contained in any flipping locus of this MMP.
Assume the sequence is not eventually disjoint from the generalized log canonical centers.
Then, there exists a generalized log canonical center of $(X/Z,B+M)$ that is intersected infinitely many times by some flipping loci 
but it is never contained in any of them.
We let $V$ be a generalized log canonical center with this condition that is minimal with respect to inclusion.
Hence, the flipping loci of the sequence of flips for $(X/Z,B+M)$ intersect the strict transform of $V$ 
infinitely many times, but they are eventually disjoint from the glc centers strictly contained in $V$.
By Proposition~\ref{qflipslcp}, 
we obtain an infinite sequence of strict ample $(K_V+B_V+M_V)$-quasi-flips. 
By~\cite{Fil20}*{Theorem 1.6}, all such quasi-flips are eventually disjoint from the glc centers of $(V,B_V+M_V)$.
The variety $V$ has dimension at most three.

Thus, up to re-indexing to omit the flips that induce the identity on $V$ (see Proposition~\ref{qflipslcp}), we obtain an infinite sequence
\begin{equation}\nonumber
 \xymatrix{
  (V/Z,B_V+M_V)\ar@{-->}^-{\pi_{V,1}}[r] & (V_1/Z,B_{V_1}+M_{V_1})\ar@{-->}^-{\pi_{V,2}}[r] & (V_2/Z,B_{V_2}+M_{V_2})\ar@{-->}^-{\pi_{V,3}}[r] & \cdots \ar@{-->}^-{\pi_{V,i}}[r] & (V_i/Z,B_{V_i}+M_{V_i})\ar@{-->}^-{\pi_{V,i+1}}[r] &\\ 
 }
\end{equation}
of strict ample $(K_V+B_V+M_V)$-quasi-flips,
which are eventually disjoint from the glc centers of $(V,B_V+M_V)$.
In particular, up to truncating the sequence, these are generalized klt quasi-flips.
By Proposition~\ref{adj}, 
we know that this sequence of quasi-flips is under a DCC set $\Omega$
and 
$M'_V =\sum_j \omega_j M'_{V,j}$, where each $M'_{V,j}$ is Cartier nef over $Z$ and $\omega_j \in \Omega$.
Here, $M'_V$ is the trace of the b-divisor associated to $M_V$ on a model where it descends.
If $V$ has dimension at most two, since the sequence of flips is eventually disjoint from the generalized log canonical centers, this sequence terminates by Lemma~\ref{tercod1}.
Hence, we may assume $V$ is a $3$-fold.

Denote by $\psi \colon V \rightarrow W$ the flipping contraction of $\pi_{V,1}$. 
By Theorem~\ref{thm:gen-dlt-mod-4}, we may consider $(V'/Z,B_{V'}+M_{V'})$ a $\qq$-factorial gdlt modification of
$(V/Z,B_{V}+M_{V})$.
Denote by
$\phi\colon V'\rightarrow V$ be the corresponding morphism.
We claim that the diminished base locus of $K_{V'}+B_{V'}+M_{V'}$ over $W$ does not intersect the generalized log canonical centers of $(V'/Z,B_{V'}+M_{V'})$.
Indeed, we have that
the diminished base locus of 
$K_V+B_V+M_V$ over $W$ equals the flipping locus, which is disjoint from the glc centers of $(V/Z,B_V+M_V)$.
By Lemma~\ref{lem:dim-pull-back} applied to the union $L$ of the glc centers of $(V/Z,B_V+M_V)$, we conclude that the diminished base locus of 
$K_{V'}+B_{V'}+M_{V'}=\pi^*(K_V+B_V+M_V)$ over $W$ is disjoint from the glc centers
of $(V'/Z,B_{V'}+M_{V'})$.

We may run a minimal model program for the generalized dlt pair $(V'/Z,B_{V'}+M_{V'})$ with scaling of an ample divisor over $W$.
We claim that this minimal model program terminates with a good minimal model
$(V'_1/Z, B_{V_1'}+M_{V_1'})$ over $W$.
By Proposition~\ref{prop:FT},
the morphism $V'\rightarrow W$ is a Fano type morphism. 
Thus, we know that this minimal model program terminates with a good minimal model.
Furthermore,
$K_{V_1}+B_{V_1}+M_{V_1}$ is the ample model for $K_{V'_1}+B_{V'_1}+M_{V'_1}$ over $W$.
Therefore, $K_{V'_1}+B_{V'_1}+M_{V'_1}$ is semiample over $W$ and we have a projective birational morphism $V'_1\rightarrow V_1$.
Note that the diminished base locus of $(V'/Z,B_{V'}+M_{V'})$ over $W$ is disjoint from $\lfloor B_{V'}\rfloor$. Then, the above sequence of flips is also a sequence of flips for the generalized klt pair
\[
(V'/Z,B_{V'}-\epsilon \lfloor B_{V'}\rfloor+M_{V'}).
\]
Thus, we replaced the strict ample quasi-flip $\pi_{V,1}$ with a finite sequence of flips for a generalized klt $3$-fold.
Proceeding inductively, 
the above infinite sequence of strict ample quasi-flips for generalized dlt $3$-folds 
induces an infinite sequence of flips for generalized klt $3$-folds.
By Proposition~\ref{ter3folds}, we conclude that the sequence must terminate.
\end{proof}

\section{\texorpdfstring{Termination of flips for $\qq$-factorial pseudo-effective dlt 4-folds}{Termination of flips for Q-factorial pseudo-effective dlt 4-folds}}\label{S5}

In this section, we prove our main theorem.
First, we prove the termination of flips for $\qq$-factorial pseudo-effective dlt $4$-folds.
The idea is to use the approach of Birkar of
computing log canonical thresholds with respect to effective divisors~\cite{Bir07}.
The main difference is that we compute the log canonical threshold with respect to some generalized boundary induced by a minimal model.
Using ACC for generalized 
log canonical thresholds, the existence of gdlt modifications, and the special termination for gdlt $4$-fold flips, we will prove Proposition~\ref{term-dlt}.
We will need the following lemma.

\begin{lemma}\label{existence-min-model}
Let $(X/Z,B+M)$ be a
$4$-dimensional
NQC pseudo-effective glc pair.
Then, we can write $K_X+B+M\equiv_{Z} N+P$,
such that 
$N$ is an effective divisor
and 
$P$ is a NQC b-nef divisor.
\end{lemma}

\begin{proof}
In~\cite{Sho09}*{Corollary 2}, it is proved that every pseudo-effective log canonical $4$-fold has a minimal model.
Then, the statement follows from~\cite{LT19}*{Corollary 3.4}.
\end{proof}

\begin{proposition}\label{term-dlt}
Let $(X/Z,B+M)$ be a $\qq$-factorial
NQC 
pseudo-effective gdlt $4$-fold.
Then, any sequence of flips for $(X/Z,B+M)$ over $Z$ terminates.
\end{proposition}

\begin{proof}
Let  
\begin{equation}\nonumber
 \xymatrix{
  (X/Z,B+M)\ar@{-->}^-{\pi_{1}}[r] & (X_1/Z,B_1+M_1)\ar@{-->}^-{\pi_{2}}[r] & (X_2/Z,B_2+M_2)\ar@{-->}^-{\pi_3}[r] & \dots \ar@{-->}^-{\pi_{i}}[r] & (X_i/Z,B_i+M_i)\ar@{-->}^-{\pi_{i+1}}[r] & \dots   \\ 
 }
\end{equation}
be a sequence of flips for the $\qq$-factorial pseudo-effective dlt $4$-fold.
We proceed by contradiction.
Assume that this sequence is infinite.
By Lemma~\ref{existence-min-model}, 
we can write
\[
K_X+B+M\equiv_{Z} N+P, 
\]
where: 
\begin{enumerate}
    \item $N$ is an effective $\rr$-divisor, and 
    \item $P$ is the trace on $X$ of a NQC
     divisor ${\bf P}$ which is b-nef over $Z$.
\end{enumerate} 
For each $X_i$, we denote by $N_i$ the push-forward of $N$ to $X_i$.
For each $X_i$, we denote by $P_i$ the trace
of ${\bf P}$ on $X_i$.
Hence, we have that $K_{X_i}+B_i+M_i\equiv_{Z} N_i+P_i$
for each $i$.
By Proposition~\ref{prop:finiteness}, we know that
\[
\lambda_{i,1}:={\rm glct}((X_i/Z,B_i+M_i);N_i+P_i)
\]
is finite unless $K_{X_i}+B_i+M_i$ is nef over $Z$.
Thus, we may assume that each $\lambda_{i,1}$ is finite.
We claim that $\lambda_{i+1,1}\geq \lambda_{i,1}$.
Since $N_i+P_i\equiv_{Z} K_{X_i}+B_i+M_i$, 
we have that $(X_i/Z,B_i+\lambda_{i,1}N_i + (M_i+\lambda_{i,1}P_i))$ is a $\qq$-factorial generalized log canonical pair and $\pi_{i+1}$ is a flip for this generalized pair.
Indeed, we have that 
\[
K_{X_i}+B_i+\lambda_{i,1}N_i+(M_i+\lambda_{i,1}P_i) \equiv_{Z} (\lambda_{i,1}+1)(K_{X_i}+B_i+M_i).
\]
Hence, $(X_{i+1}/Z,B_{i+1}+\lambda_{i,1}N_{i+1}+ 
(M_{i+1}+\lambda_{i,1}P_{i+1}))$ is again generalized log canonical, proving the claim.
By Proposition~\ref{monotonicity}, 
if $\lambda_{i+1,1}=\lambda_{i,1}$, then $\pi_{i+1}$ does not contain all the 
glc centers of $(X_i/Z,B_i+\lambda_{i,1} N_i+ (M_i+\lambda_{i,1} P_i))$
in its flipping locus.\\

\noindent \textit{Claim:} There exists a positive number $\lambda$ satisfying the following conditions:
\begin{enumerate}
    \item eventually, all flipping loci are contained in the generalized log canonical locus of $(X_i/Z,B_i+\lambda N_i+ (M_i+\lambda P_i))$, and
    \item for $i$ large enough, the generalized pairs $(X_i/Z,B_i+\lambda N_i +(M_i+\lambda P_i))$ have a glc center that is intersected by infinitely many flipping loci.
\end{enumerate}

\begin{proof}[Proof of the Claim]
Assume that the two above conditions are never satisfied. We aim to contradict
the ascending chain conditions for generalized log canonical thresholds~\cite[Theorem 1.5]{BZ16}. 
To do so, we will construct a sequence of generalized log canonical thresholds associated to this sequence of flips. 
Note that for every $\lambda>0$, the non-glc locus of each model
$(X_i/Z,B_i+\lambda N_i+(M_i+\lambda P_i))$ is closed, the proof
being verbatim to such of~\cite[Corollary 2.31]{KM98}.
We let $\lambda_{\infty,1}$ to be the limit of the $\lambda_i$'s.
By~\cite{BZ16}*{Theorem 1.5}, for $i$ large enough, we will have that 
$\lambda_{i,1}=\lambda_{\infty,1}$.
Note that $(1)$ is satisfied for $\lambda_{\infty,1}$.
By Proposition~\ref{monotonicity},
there exists $i_1$ satisfying the following:
for $i\geq i_1$, the generalized pairs
$(X_i/Z,B_i+\lambda_{\infty,1}N_i+(M_i+\lambda_{\infty,1}P_i))$
are not gklt and their glc centers stabilize.
This means that
the glc centers of $(X_{i+1}/Z,B_{i+1}+\lambda_{\infty,1}N_i+(M_{i+1}+\lambda_{\infty,1}P_i))$ are the strict transforms of the glc centers
of $(X_i/Z,B_i+\lambda_{\infty,1}N_i+(M_i+\lambda_{\infty,1}P_i))$.
If $(2)$ does not hold, then eventually all flips are disjoint from the glc centers
of $(X_i/Z,B_i+\lambda_{\infty,1}N_i+(M_i+\lambda_{\infty,1}P_i))$.
This implies that for $i\gg 0$, the gklt locus of both 
\[
(X_i/Z,B_i+\lambda_{\infty,1}N_i+(M_i+\lambda_{\infty,1}P_i))
\text{ and }
(X_{i+1}/Z,B_{i+1}+\lambda_{\infty,1}N_i+(M_{i+1}+\lambda_{\infty,1}P_i))
\]
maps properly onto its image in $W_i$. 
Up to increasing $i_1$, we may assume that for every $i\geq i_1$, the flipping locus of
$\pi_{i+1}$ is disjoint from the glc centers
of $(X_i/Z,B_i+\lambda_{\infty,1}N_i+(M_i+\lambda_{\infty,1}P_i))$.
By Proposition~\ref{prop:finiteness}, for $i\geq i_1$, the generalized log canonical threshold of $(X_i/Z,B_i+M_i)$ with respect to $N_i+P_i$ on the gklt locus
of $(X_i/Z,B_i+\lambda_{\infty,1}N_i+(M_i+\lambda_{\infty,1}P_i))$ is finite.
Furthermore, this glct is larger than $\lambda_{\infty,1}$.

For $k\geq 2$, we define $\lambda_{\infty,k}$ and $i_k$ inductively.
The numbers $\lambda_{\infty,1}$ and $i_1$ are already defined.
For every $i\geq i_{k-1}$, we define 
$\lambda_{i,k}$ to be the generalized log canonical threshold of $(X_i/Z,B_i+M_i)$ with respect to $N_i+P_i$ on the gklt locus of $(X_i/Z,B_i+\lambda_{\infty,k-1}N_i+(M_i+\lambda_{\infty,k-1}P_i))$.
By induction, these generalized log canonical thresholds are finite and larger than $\lambda_{\infty,k-1}$.
By construction, the $\lambda_{i,k}$ forms a non-decreasing sequence on $i$ and they satisfy the ascending chain condition.
We let $\lambda_{\infty,k}$ to be the limit of the $\lambda_{i,k}$'s. 
Observe that $\lambda_{\infty,k}>\lambda_{\infty,k-1}$. By~\cite{BZ16}*{Theorem 1.5}, for $i$ large enough, we have that
$\lambda_{i,k}=\lambda_{\infty,k}$.
The Condition (1) is satisfied for the generalized pairs $(X_i/Z,B_i+\lambda_{\infty,k}N_i+(M_i+\lambda_{\infty,k}P_i))$. 
Indeed, by the definition of $\lambda_{\infty,k}$, the non-glc centers of the pair $(X_i/Z,B_i+\lambda_{\infty,k}N_i+(M_i+\lambda_{\infty,k}P_i))$ are contained in the non-gklt centers of $(X_i/Z,B_i+\lambda_{\infty,k-1}N_i+(M_i+\lambda_{\infty,k-1}P_i))$,
from which the flip $\pi_{i+1}$ are disjoint.
By Proposition~\ref{monotonicity},
there exists $i_k$ satisfying the following.
For $i\geq i_k$, 
the glc centers of 
$(X_i/Z,B_i+\lambda_{\infty,k}N_i+(M_i+\lambda_{\infty,k}P_i))$
that are not contained 
in the non-glc locus of
$(X_i/Z,B_i+\lambda_{\infty,k-1}N_i+(M_i+\lambda_{\infty,k-1}P_i))$ stabilize.
If $(2)$ does not hold, then eventually each flip $\pi_{i+1}$ is disjoint from the non-gklt locus of 
$(X_i/Z,B_i+\lambda_{\infty,k}N_i+(M_i+\lambda_{\infty,k}P_i))$.
Up to increasing $i_k$, we may assume that for every $i\geq i_k$, the flipping locus of $\pi_{i+1}$ is disjoint from the non-gklt locus
of $(X_i/Z,B_i+\lambda_{\infty,k}N_i+(M_i+\lambda_{\infty,k}P_i))$.
In particular, for $i\gg 0$, the gklt locus of both 
\[
(X_i/Z,B_i+\lambda_{\infty,k}N_i+(M_i+\lambda_{\infty,k}P_i)) 
\text{ and }
(X_{i+1}/Z,B_{i+1}+\lambda_{\infty,k}N_{i+1}+(M_{i+1}+\lambda_{\infty,k}P_{i+1}))
\]
maps properly onto its image in $W_i$. 
By Proposition~\ref{prop:finiteness}, for $i\geq i_k$, the generalized log canonical threshold of $(X_i/Z,B_i+M_i)$ with respect to $N_i+P_i$ on the gklt locus of
$(X_i/Z,B_i+\lambda_{\infty,k}N_i+(M_i+\lambda_{\infty,k}P_i))$ is finite.

Proceeding inductively, we obtain an infinite increasing sequence of log canonical thresholds
\[
\lambda_{\infty,1}<\lambda_{\infty,2}<\dots<
\lambda_{\infty,k}<\dots
\]
This contradicts~\cite{BZ16}*{Theorem 1.5}.
We conclude that such $\lambda$ exists.
\end{proof}

Let $\lambda$ be as in the claim.
Truncating our sequence of flips, we have a sequence of generalized log canonical flips for the generalized pair $(X/Z,B+\lambda N + (M+\lambda P))$.
The previous statement follows from the negativity of the MMP. 
By Proposition~\ref{monotonicity}, we know that for $i\gg 0$ the generalized log canonical centers of 
$(X_i/Z,B_i+\lambda N_i +(M+\lambda P_i))$ contained in the glc locus, stabilize. 
Further, for $i\gg 0$, the flipping locus of $\pi_i$ is disjoint from the non-glc locus 
of $(X_i/Z,B_i+\lambda N_i + (M+\lambda P_i))$, which is a closed subset. 
By Proposition~\ref{thm:gen-dlt-mod-4}, we can take a $\qq$-factorial generalized dlt modification $(Y/Z,B_Y+M_Y)$ of $(X/Z,B+\lambda N +(M+\lambda P))$.
Let $\phi \colon X \rightarrow W$ be the flipping contraction of $\pi\colon X \dashrightarrow X_1$.
We denote by $X^0$ the glc locus
of $(X/Z,B+\lambda N+(M+\lambda P))$
and by $W^0$ its image on $W$.
We denote by $Y^0$ the preimage of $X^0$ in $Y$.
We argue that the morphism $Y^0\rightarrow W^0$ is projective. 
The non-glc locus $K$ of $(X/Z,B+\lambda N+(M+\lambda P))$
is closed and is disjoint from the exceptional locus of $X\rightarrow W$.
Then, we have that $X^0=X\setminus K$ is proper over $W^0=W\setminus \phi(K)$. 
Therefore, the morphism $Y^0\rightarrow W^0$ is of Fano type by Proposition~\ref{prop:FT}.
Let $B_{Y^0}$ (resp. $M_{Y^0}$)
be the restriction of $B_Y$ (resp. $M_Y$) to $Y^0$.
We run a $(K_Y+B_Y+M_Y)$-MMP over $W$.
This is also a MMP for
$K_{Y^0}+B_{Y^0}+M_{Y^0}$ over $W^0$ 
as ${\rm Ex}(\phi)$ is contained in $X^0$.
Hence, it must terminate as $Y^0$ is of Fano type over $W^0$.
Let $Y_1$ be the model over $W$ where this MMP terminates.
Denote by $B_{Y_1}$ (resp. $M_{Y_1}$) 
the push-forward of $B_Y$ (resp. $M_Y$) to $Y_1$.
Note that $X_1$ is the ample model for $K_{Y_1}+B_{Y_1}+M_{Y_1}$ over $W^0$.
Then, by taking the ample model
of $K_{Y_1}+B_{Y_1}+M_{Y_1}$ over $W$,
we obtain a birational morphism $Y_1\rightarrow X_1$ which is a gdlt modification of
$(X_1,B_1+\lambda N_1 + (M+\lambda P_1))$
over $X_1^0$. 
In the previous statement, we have used the fact that $Y^1$ is proper over $W^0$. 
Here, $X_1^0$ is the glc locus 
of $(X_1/Z,B_1+\lambda N_1 + (M_1+\lambda P_1))$.

We let $\phi_i\colon X_i\rightarrow W_i$
be the flipping contraction
of $\pi_i\colon X_i\dashrightarrow X_{i+1}$.
Inductively, for each $(Y_i/Z,B_{Y_i}+M_{Y_i})$, we may run a minimal model program over $W_i$
which terminates with a good minimal model
$(Y_{i+1}/Z,B_{i+1}+M_{i+1})$ over $W_i$.
Moreover, $Y_{i+1}\dashrightarrow X_{i+1}$ is
a gdlt modification on the glc locus of
$(X_{i+1}/Z,B_{i+1}+\lambda N_{i+1}+ (M_{i+1}+\lambda P_{i+1}))$.
Thus, we obtain a sequence of flips for the $\qq$-factorial gdlt pair $(Y/Z,B_Y+M_Y)$.
By construction, infinitely many flips intersect
$\lfloor B_Y\rfloor$.
This contradicts Proposition~\ref{special-term}.
\end{proof}

\begin{corollary}\label{cor:term-gdlt}
Let $(X/Z,B+M)$ be a
$4$-dimensional
$\qq$-factorial
NQC 
pseudo-effective 
generalized dlt pair.
Then, any minimal model program for $(X/Z,B+M)$ over $Z$ terminates.
\end{corollary}

\begin{proof}
The Picard rank of $X$ relative to $Z$ decreases after a divisorial contraction 
and stays equal after a flip.
We conclude that this MMP can only have finitely many divisorial contractions.
Thus, it suffices to prove that every sequence of flips is finite.
This is the content of Proposition~\ref{term-dlt}.
\end{proof}

\begin{theorem}\label{termination-gen}
Let $(X/Z,B+M)$
be a $4$-dimensional 
NQC 
pseudo-effective 
generalized lc pair.
Then, any minimal model program for $(X/Z,B+M)$ over $Z$ terminates.
\end{theorem}

\begin{proof}
Let 
\begin{equation}\nonumber
 \xymatrix{
  (X/Z,B+M)\ar@{-->}^-{\pi_{1}}[r] & (X_1/Z,B_1+M_1)\ar@{-->}^-{\pi_{2}}[r] & (X_2/Z,B_2+M_2)\ar@{-->}^-{\pi_3}[r] & \dots \ar@{-->}^-{\pi_{i}}[r] & (X_i/Z,B_i+M_i)\ar@{-->}^-{\pi_{i+1}}[r] & \dots   \\ 
 }
\end{equation}
be a minimal model program for $(X/Z,B+M)$ over $Z$.
The number of generalized log canonical centers of $(X/Z,B+M)$ is finite by Corollary~\ref{cor:finiteness-glcc}.
After finitely many steps, we may assume that the generalized log canonical centers stabilize.
We may truncate the sequence and assume this is the case.
Let $\phi \colon X\rightarrow W$ be the flipping contraction of $\pi_1$.
Let $(Y/Z,B_Y+M_Y)$ be a $\qq$-factorial gdlt modification of $(X/Z,B+M)$.
By~\cite{BZ16}*{Theorem 4.4},
we can run a minimal model program for $(Y/Z,B_Y+M_Y)$ over $W$ with scaling of an ample divisor over $W$ wich terminates with a good minimal model. 
Let $(Y_1/Z,B_{Y_1}+M_{Y_1})$ be its good minimal model over $W$. Then, $(X_1/Z,B_1+M_1)$ is the ample model of $(Y_1/Z,B_{Y_1}+M_{Y_1})$ over $W$.
Proceeding inductively, we obtain a minimal model program 
\begin{equation}\nonumber
 \xymatrix{
  (Y/Z,B_Y+M_Y)\ar@{-->}^-{\pi_{1}}[r] & (Y_1/Z,B_{Y_1}+M_{Y_1})\ar@{-->}^-{\pi_{2}}[r] & \dots \ar@{-->}^-{\pi_{i}}[r] & (Y_i/Z,B_{Y_i}+M_{Y_i})\ar@{-->}^-{\pi_{i+1}}[r] & \dots   \\ 
 }
\end{equation}
for a $\qq$-factorial generalized dlt pair $(Y/Z,B_Y+M_Y)$ of dimension $4$.
Furthermore,
$K_Y+B_Y+M_Y$ is pseudo-effective over $Z$.
Hence, by Corollary~\ref{cor:term-gdlt}, we conclude that the minimal model program for the $(Y_i/Z,B_{Y_i}+M_{Y_i})$'s terminates. 
Thus, for $i\gg 0$ we have $a_E(Y_i/Z,B_{Y_i}+M_{Y_i})=a_E(Y_{i+1}/Z,B_{Y_{i+1}}+M_{Y_{i+1}})$ for every $E$. 
As each $(Y_i/Z,B_{Y_i}+M_{Y_i})\rightarrow (X_i/Z,B_i+M_i)$ crepant, 
for $i\gg 0$ we have
$a_E(X_i/Z,B_i+M_i)=a_E(X_{i+1}/Z,B_{i+1}+M_{i+1})$ for every $E$.
Thus, Proposition~\ref{monotonicity} implies that the sequence of flips
for $(X/Z,B+M)$ terminates after finitely many steps. 
\end{proof}

\begin{proof}[Proof of Theorem~\ref{termination}]
This follows from Theorem~\ref{termination-gen}.
\end{proof}

\begin{proof}[Proof of Theorem~\ref{thm:min-models}]
This follows from Theorem~\ref{termination-gen}.
\end{proof}

\bibliographystyle{habbvr}
\bibliography{bib}

\end{document}